\documentclass[11pt,reqno]{amsart}

\usepackage{verbatim, latexsym, amssymb, amsmath}
\usepackage{tensor}

\newtheorem{thm}{Theorem}
\newtheorem{lemm}[thm]{Lemma}

\newtheorem{prop}[thm]{Proposition}
\newtheorem*{thmA}{Theorem A}
\newtheorem*{thmB}{Theorem B}
\newtheorem*{thmC}{Theorem C}
\newtheorem*{thmD}{Theorem D}
\newtheorem*{thmE}{Theorem E}

\theoremstyle{remark}
\newtheorem{rmk}{Remark}

\theoremstyle{definition}
\newtheorem*{defi}{Definition}

\title[On static three-manifolds with positive scalar curvature]{On static three-manifolds with \\ positive scalar curvature}
\author{Lucas Ambrozio}
\address{\noindent Imperial College London, South Kensington Campus, London, UK \newline \indent \textit{E-mail} : l.ambrozio@imperial.ac.uk}
\thanks{The author was supported by CNPq-Brazil}

\begin{document}
 
\begin{abstract}
We compute a Bochner type formula for static three\--mani\-folds and deduce some applications in the case of positive scalar curvature. We also explain in details the known general construction of the (Riemannian) Einstein $(n+1)$-manifold associated to a maximal domain of a static $n$-manifold where the static potential is positive. There are examples where this construction inevitably produces an Einstein metric with conical singularities along a codimension-two submanifold. By proving versions of classical results for Einstein four-manifolds for the singular spaces thus obtained, we deduce some classification results for compact static three-manifolds with positive scalar curvature.
\end{abstract}

\maketitle

\section{Introduction and statements of the main results}

\indent Let $(M^{n},g)$ be a complete Riemannian manifold. A \textit{static potential} is a non-trivial solution $V\in C^{\infty}(M)$ to the equation
\begin{equation} \label{introeqestatica}
  Hess_{g} V - \Delta_g V g - VRic_{g} = 0.
\end{equation}
\indent Riemannian manifolds admitting static potentials are very interesting geometric objects which arise in different contexts. For instance, the left hand side of (\ref{introeqestatica}) defines the formal adjoint of the linearisation of the scalar curvature and the exis\-tence of static potentials plays an important role in problems related to prescribing the scalar curvature function (see \cite{FisMar}, \cite{Bour} and also \cite{Cor}). On the other hand, given a solution to (\ref{introeqestatica}) in a three-manifold it is possible to construct a space-time satisfying the vacuum Einstein equations (with cosmological constant), whose properties, physically interpreted, justify the name \text{static} (see, for example, \cite{Wal}).  \\
\indent The existence of a static potential imposes many restrictions on the geometry of the underlying manifold. For instance, its scalar curvature must be constant and the set of zeroes of $V$ is a totally geodesic regular hypersurface \cite{Bour}. It might seem reasonable to believe that a description of all Riemannian manifolds admitting a static potential is possible. In dimension $n=3$, a solution to this classification problem seems to be more likely (because the Ricci tensor determines completely the Riemann curvature tensor) and has also physical interest (since $n=3$ is the relevant dimension in the General Theory of Relativity).
\begin{defi}
A \textit{static triple} is a triple $(M^n,g,V)$ consisting of a connected $n$-dimen\-sio\-nal smooth manifold $M$ with boundary $\partial M$ (possibly empty), a complete Riemannian metric $g$ on $M$ and a static potential $V\in C^{\infty}(M)$ that is \textit{non-negative and vanishes precisely on} $\partial M$. Two static triples $(M_{i},g_{i},V_{i})$, $i=1,2$, are said to be \textit{equivalent} when there exists a diffeomorphism $\phi : M_{1} \rightarrow M_{2}$ such that $\phi^{*}g_2 = c g_1$ for some constant $c>0$ and $V_2\circ \phi = \lambda V_1$ for some constant $\lambda>0$.
\end{defi}

\indent In other words, a static triple is a maximal connected domain where a static potential is positive, and two static triples are equivalent when they are isometric (up to scaling) with proportional static potentials. The classification problem is to describe all equivalence classes of static triples. \\
\indent The literature on the subject is extensive and many classification results are known particularly in the zero and negative scalar curvature cases (for example, see \cite{Isr}, \cite{HagRobSei}, \cite{Rob}, \cite{BunMas}, \cite{And} and \cite{BouGibHor}, \cite{Wan}, \cite{Qin}, \cite{AndChrDel}, \cite{GalSurWoo}, \cite{ChrSim}, \cite{LeeNev}). \\
\indent In this work we focus on three-dimensional static triples with positive scalar curvature. \\
\indent Fischer and Marsden \cite{FisMar} conjectured that the standard unit round spheres $(S^n,g_{can})$ are the only closed Riemannian manifolds with scalar curvature $n(n-1)$ admitting static potentials. In fact, the linear combinations of the coordinate functions are all solutions to (\ref{introeqestatica}) in $(S^n,g_{can})$. This conjecture was soon realized to be too optimistic. Kobayashi \cite{Kob} and Lafontaine \cite{Laf} proved more generally a classification result for locally conformally flat static triples in all dimensions. We state their result in dimension $n=3$ as follows.
\begin{thm}[Kobayashi \cite{Kob}, Lafontaine \cite{Laf}] \label{thmconfflat}
 Let $(M^3,g,V)$ be a static triple with positive scalar curvature. If $(M^3,g)$ is locally conformally flat, then $(M^3,g,V)$ is covered by a static triple that is equivalent to one of the following static triples:
 \begin{itemize}
  \item[$i)$] The standard round hemisphere,
  \begin{equation*}
   \left(S^3_{+} , g_{can} , V = x_{n+1}\right).
  \end{equation*}
  \item[$ii)$] The standard cylinder over $S^2$ with the product metric,
  \begin{equation*}
   \left( \left[0,\frac{\pi}{\sqrt{3}}\right] \times S^2 , g_{prod}=dt^2 + \frac{1}{3} g_{can}, V = \frac{1}{\sqrt{3}} \sin(\sqrt{3} t) \right).
  \end{equation*} 
  \item[$iii)$] For some $m\in (0,\frac{1}{3\sqrt{3}})$, the triple
  \begin{equation*}
   \left( [r_h(m),r_{c}(m)]\times S^2 , g_m = \frac{dr^2}{1-r^2-\frac{2m}{r}} + r^2 g_{can}, V_m = \sqrt{1-r^2-\frac{2m}{r}}    \right),
  \end{equation*}
  \noindent where $r_h(m) < r_c(m)$ are the positive zeroes of $V_m$.
 \end{itemize}
\end{thm}
\begin{rmk}
 A static triple $(\tilde{M}^n,\tilde{g},\tilde{V})$ covers a static triple $(M^n,g,V)$ when there exists a  covering map $\pi: \tilde{M}\rightarrow M$ such that $\tilde{g}=\pi^{*}g$ and $\tilde{V}=V\circ\pi$.
\end{rmk}
\begin{rmk}
 We have normalized the scalar curvature of the examples above to be $6$. They are sometimes referred as (time-symmetric slices of) the \textit{de Sitter space}, the \textit{Nariai space} and the \textit{Schwarzschild-de Sitter spaces of positive mass $m$}, respectively (see \cite{BouGibHor}). 
\end{rmk}
\begin{rmk}
 The three types of locally conformally flat manifolds described above can be distinguished from the behaviour of their Ricci tensor. The round sphere is an Einstein manifold, the standard cylinder has parallel Ricci tensor but is not Einstein, and the Schwarzschild-de Sitter spaces of positive mass are locally conformally flat but do not have parallel Ricci tensor.
\end{rmk}
\indent In dimension $n=3$, the above list contains all known examples of compact simply connected static triples with positive scalar curvature. Qing and Yuan \cite{QinYua} proved uniqueness of these examples assuming a weaker hypothesis on the Cotton tensor. \\
\indent It is interesting to observe that precisely two more examples can be obtained as quotients of the standard cylinder, one of them \textit{non-orientable} with \textit{two} boundary components, the other one \textit{orientable} with \textit{connected} boundary (see their description in the end of Section 7). One should think about this last example in the context of the ``cosmic no-hair conjecture'' suggested in \cite{BouGibHor} and correctly stated in \cite{BouGib}. The existence of examples with more than two boundary components was also speculated in \cite{BouGibHor}. In any case, the question remains whether the standard round hemisphere is the only compact \textit{simply connected} static triple $(M^3,g,V)$ with positive scalar curvature and \textit{connected} boundary.  \\
\indent An important result towards the answer to this question states that the standard hemisphere has the maximum possible boundary area among static triples with positive scalar curvature and connected boundary. More precisely, the following theorem holds.
\begin{thm}[Boucher-Gibbons-Horowitz \cite{BouGibHor}, Shen \cite{Shen}] \label{thm4pi}
 Let $(M^3,g,V)$ be a compact oriented static triple with connected boundary and scalar curvature $6$. Then $\partial M$ is a two-sphere whose area satisfies the inequality
 \begin{equation*}
  |\partial M| \leq 4\pi.
 \end{equation*}
 \indent Moreover, equality holds if and only if $(M^3,g,V)$ is equivalent to the standard hemisphere.
\end{thm}
\indent This theorem was also generalized to allow more boundary components (for the precise statement, see Section 2, Proposition \ref{propshenform}). Hijazi, Montiel and Raulot \cite{HijMonRau} established a similar result that involves a inequality for the first eigenvalue of the induced Dirac operator of each boundary component. \\
\indent Theorem \ref{thm4pi} has some remarkable interpretations. The inequality for the area of the boundary can be thought as a Penrose like inequality in the context of positive scalar curvature (as suggested in \cite{BouGibHor}). The rigidity phenomenon associated to the equality case is related to Min-Oo's Conjecture, which is known to be false after the work of Brendle, Marques and Neves \cite{BreMarNev}. In view of that, Theorem \ref{thm4pi} gives an interesting rigidity result for the standard round hemisphere under a nice extra geometric assumption (for related results, see for example \cite{HanWan}, \cite{Eic} and \cite{MarNev}). \\
\indent In this paper we prove some classification results for static three-ma\-ni\-folds with positive scalar curvature. In particular, we obtain a stronger version of Theorem \ref{thm4pi} that includes the oriented static quotient of the standard cylinder in the picture. Our results may also be used to rule out some possible new examples of compact static triples with positive scalar curvature. \\
\indent The main ingredient is a Bochner type formula, which holds for any three-dimensional static triple. Let $R$ denote the scalar curvature, $\mathring{Ric}=Ric-(R/3)g$ denote the trace-free part of the Ricci tensor and $C$ denote the Cotton tensor of $(M^3,g)$. $C$ is the $3$-tensor defined by
\begin{multline*}
 C(X,Y,Z) = (\nabla_{Z}Ric)(X,Y) - (\nabla_{Y}Ric)(X,Z) \\ - \frac{1}{4}(dR(Z)g(X,Y) - dR(Y)g(X,Z))
\end{multline*}
\noindent for all $X,Y,Z\in \mathcal{X}(M)$. A classical result asserts that $(M^3,g)$ is locally conformally flat if and only if its Cotton tensor vanishes (see \cite{Aub}, Chapter 4, Section 3.2). \\
\indent If $(M^3,g)$ admits a static potential $V$, we prove the following identity:
\begin{equation} \label{eqBochner}
 \frac{1}{2}div(V\nabla|\mathring{Ric}|^2) = \left(|\nabla\mathring{Ric}|^2 + \frac{|C|^2}{2}\right)V + \left(R |\mathring{Ric}|^2 + 18 det(\mathring{Ric})\right)V.
\end{equation}
\indent Assuming a pointwise inequality for the traceless Ricci tensor we deduce the first consequence of formula (\ref{eqBochner}).
\begin{thmA} \label{thmBochner}
 Let $(M^3,g,V)$ be a compact oriented static triple with positive scalar curvature. If 
 \begin{equation*}
  |\mathring{Ric}|^2 \leq \frac{R^2}{6},
 \end{equation*}
 \noindent then one of the following alternatives holds:
 \begin{itemize}
  \item[$i)$] $\mathring{Ric}= 0$ and $(M^3,g,V)$ is equivalent to the standard hemisphere; or
  \item[$ii)$] $|\mathring{Ric}|^2 = R^2/6$ and $(M^3,g,V)$ is covered by a static triple that is equivalent to the standard cylinder.
 \end{itemize}
\end{thmA}
\indent In the next applications we work with the \textit{singular} Einstein manifold that can be constructed from a static triple. The general construction has already been discussed in the literature (for instance, in \cite{GibHaw} and \cite{BouGibHor}). When not singular, the obtained Riemannian manifolds are sometimes called \textit{gravitational instantons}, although this terminology might be confusing. \\
\indent As mentioned before, as a general fact, given a static triple $(M^n,g,V)$ with scalar curvature $R=\epsilon n(n-1)$, $\epsilon \in\{-1,0,1\}$, the metrics
\begin{equation*}
 h_{\pm} = \pm V^2 dt^2 + g 
\end{equation*}
\noindent  satisfy the Einstein equation $Ric_{h_{\pm}} = \epsilon n h_{\pm}$. While the Lorentzian metric $h_{-}$ has no singularities on $\mathbb{R}\times\partial M$  (as the vanishing of $V$ only means that $\partial_t$ is a light-like vector on this set), the Riemannian metric $h_{+}$ becomes singular on $\mathbb{R}\times\partial M$. One way to overcame this problem could be to identify the variable $t$ with period $2\pi$ (as done in \cite{BouGibHor} and \cite{GibHaw}). When this procedure removes the metric singularity, one obtains precisely the Einstein manifolds studied by Seshadri \cite{Ses}. However, although metric singularities might be inevitable, the singular space obtained is well-behaved: it has a cone-like singularity along a codimension-two submanifold. \\
\indent This type of singular space, sometimes called an \textit{edge space}, has been long studied in the literature and its theory presents many interesting geometric and analytical phenomena. For instance, there has been very recent progress in understanding the Yamabe problem in these spaces, see the work of Akutagawa, Carron and Mazzeo \cite{AkuCarMaz1} and Mondello \cite{Mon}. On the other hand, classical formulas like the Gauss-Bonnet-Chern formula for closed $4$-manifolds have generalizations to these spaces, see the work of Liu and Shen \cite{LiuShe} (and also the work of Atiyah and LeBrun \cite{AtiLeB}). For a comprehensive exposition on singular Riemannian spaces of stratified type, see \cite{Pfl}. \\
\indent In the second part of this paper, after describing carefully the above construction and the properties of the singular metric $h_{+}$ (see Section 6), we argue that part of the classical Bonnet-Myers Theorem and a deep theorem of Gursky \cite{Gur} on Einstein four-manifolds with positive scalar curvature hold for the singular Einstein manifolds obtained from a compact static triple with positive scalar curvature. As a consequence, we prove two results. \\
\indent The first one gives a description of the topology of compact static triples $(M^3,g,V)$ with positive scalar curvature. It is inspired by the work of Galloway \cite{Gal} on the topology of black holes and generalizes results of \cite{Ses}. Before stating it, we recall that a closed minimal surface $\Sigma^2$ in $(M^3,g)$ is called \textit{stable} when the second variation of the area is non-negative for all variations, and \textit{unstable} otherwise.
\begin{thmB}
 Let $(M^3,g,V)$ be a compact oriented static triple with positive scalar curvature. Then
 \begin{itemize}
  \item[$i)$] The universal cover of $M$ is compact.
  \item[$ii)$] If $\partial M$ contains unstable connected components, then $\partial M$ contains exactly one unstable connected component. In this case, $M$ is simply connected.
  \item[$iii)$] Each connected component of $\partial M$ is diffeomorphic to a sphere.
  \end{itemize}
\end{thmB}
\begin{rmk}
 Regarding the classification problem of compact static triples with positive scalar curvature, item $i)$ above allows to assume simply connectedness without loss of generality.
\end{rmk}
\indent The second result can be seen as a gap result for compact static triples $(M^3,g,V)$ with positive scalar curvature.
\begin{thmC}
 Let $(M^3,g,V)$ be a compact simply connected static triple with scalar curvature $6$. Then, one of the following alternatives holds:
 \begin{itemize}
  \item[$i)$] $(M^3,g,V)$ is equivalent to the standard hemisphere; or
  \item[$ii)$] $(M^3,g,V)$ is equivalent to the standard cylinder; or
  \item[$iii)$] Denoting by $\partial_{i} M$, $i=1,\ldots, r$, the connected components of $\partial M$ and by $k_{i}$ the (constant) value of $|\nabla V|$ on $\partial_{i} M$, the following inequality holds:
  \begin{equation*}
   \sum_{i=1}^{r} k_{i} |\partial_i M| < \frac{4\pi}{3}\sum_{i=1}^{r} k_{i}.
  \end{equation*}
 \end{itemize}  
\end{thmC}
\begin{rmk}
  This theorem should be compared with Theorem B in \cite{Gur}. In fact, when the associated Einstein manifold to a static triple satisfying the hypotheses of Theorem C has no singularities, the result is a direct corollary of that theorem (see Section 6 for more details).
\end{rmk}
\begin{rmk}
 The inequality in $iii)$ is invariant under rescaling of the static potential. Theorem C is sharp in the sense that one can explicitly verify that:
 \begin{itemize}
 \item[$i)$] For the standard cylinder, $k_1=k_2$ and $|\partial_1 M|=|\partial_2 M| = 4\pi/3$.
 \item[$ii)$] For the Schwarzschild-de Sitter spaces of mass $m\in(0,1/3\sqrt{3})$,
 \begin{equation*}
  \frac{k_{1} |\partial_1 M| + k_2|\partial_2 M|}{k_1+k_2}
 \end{equation*} 
 \noindent is an increasing function of $m$ which converges to $0$ as $m\rightarrow 0$ and to $4\pi/3$ as $m\rightarrow 1/3\sqrt{3}$. Moreover, $|\partial_1 M| < 4\pi/3 < |\partial_2 M|$ and $k_1 > k_2$ for all $m\in (0,1/3\sqrt{3})$.
 \end{itemize}
\end{rmk}
\indent As an immediate application, we can state
\begin{thmD}
 Let $(M^3,g,V)$ be a compact simply connected static triple with connected boundary and scalar curvature $6$. If
 \begin{equation*}
  |\partial M| \geq \frac{4\pi}{3},
 \end{equation*}
\noindent then $(M^3,g,V)$ is equivalent to the standard hemisphere. 
\end{thmD}
\indent Another corollary of the previous results is the following %(see also Remark \ref{rmkE})
\begin{thmE}
 Let $(M^3,g,V)$ be a compact static triple with scalar curvature $6$ and non-negative Ricci curvature. Then $(M^3,g,V)$ is equivalent to the standard hemisphere or is covered by the standard cylinder.
\end{thmE}
\indent In fact, in this case, one can see using the Gauss equation that each component of $\partial M$ is a totally geodesic two-sphere with Gaussian curvature $K \leq 3$, hence each component has area greater than or equal to $4\pi/3$. \\

\indent The paper is organized as follows. After reviewing some basic material on static triples (Section 2), we deduce some formulas for the Cotton tensor of such triples (Section 3) and prove the Bochner type formula (\ref{eqBochner}) (Section 4). Theorem A is proven in Section 5. In Section 6 we describe the associated singular Einstein manifold and briefly discuss some of its geometric and topological properties. In Section 7 we prove the topological classification Theorem B. The remaining two sections are devoted to the proof of Theorem C. The proof mimics the arguments of Gursky in \cite{Gur}, which involve obtaining the solution to a Yamabe type problem and the Gauss-Bonnet-Chern formula for closed oriented four-manifolds. Technical work is needed to justify why these arguments also work for the singular Einstein four-manifold associated to a compact three-dimensional static triple. The paper ends after three appendixes. The first one contains the proof of a key inequality used in Section 8. Then we classify certain solutions to an equation of type $Hess V = (\Delta V/2)g$ in a compact surface, a result needed in Section 7. The last appendix discusses the regularity properties of the solution to a degenerate elliptic problem that appeared in the proof of Proposition \ref{propfund}. \\

\noindent \textit{Acknowledgements:} I am grateful to Fernando Cod\'a Marques, Harold Rosenberg and Andr\'e Neves for their encouragement, for interesting conversations and for their kind interest in this work. I was supported by CNPq-Brazil.

%%%%%%%%%%%%%%%%%

\section{General properties of static triples}

\indent This section is intended to present some basic properties of static manifolds that will be frequently used latter (see also \cite{Cor}). \\
\indent Let $(M^{n},g)$ be a complete Riemannian manifold. A static potential is a non-trivial solution $V\in C^{\infty}(M)$ to the second order overdetermined elliptic equation
\begin{equation} \label{eqestatica}
  Hess_{g} V - \Delta_g V g - VRic_{g} = 0.
\end{equation}
\indent Equation (\ref{eqestatica}) is equivalent to a useful system of equations. Moreover, the set of zeroes of $V$ has a very special geometry.
\begin{lemm}[\cite{FisMar}, \cite{Bour}]
 The static equation (\ref{eqestatica}) is equivalent to the equations
 \begin{align}
  Hess_{g} V & = V(Ric_{g} - \Lambda g), \label{eqest2a} \\
  \Delta_{g} V + \Lambda V & = 0, \label{eqest2b}
 \end{align}
where the scalar curvature $R_g = (n-1)\Lambda$ is constant. Moreover, if $(M^n,g)$ admits a static potential $V\in C^{\infty}(M)$, then
 \begin{itemize}
  \item[$i)$] $0$ is a regular value of $V$;
  \item[$ii)$] $\{V=0\}$ is totally geodesic; and
  \item[$iii)$] $|\nabla V|$ is locally constant and positive on $\{V=0\}$.
 \end{itemize}
\end{lemm}
\begin{proof} (see \cite{Cor}) Taking the trace of (\ref{eqestatica}), we obtain
 \begin{equation} \label{eqcomp1}
  \Delta_{g} V + \frac{R_g}{n-1}V = 0.
 \end{equation}
 \indent The equivalence between (\ref{eqestatica}) and the system (\ref{eqest2a}), (\ref{eqest2b}) follows immediately. To prove that $R_g$ is constant, we observe that the divergence of $(\ref{eqestatica})$ gives 
 \begin{equation} \label{eqcomp2}
  V dR_g = 0,
 \end{equation}
 \noindent as can be seen by the Ricci formula for commuting derivatives and the contracted second Bianchi identity $2div_g Ric_g = dR_g$. Considering the homogeneous ODE satisfied by the restriction of $V$ to geodesics (which follows from (\ref{eqest2a})), we conclude that if $\nabla V(p)=0$ at $p\in\{V=0\}$, then $V$ must vanishes identically near $p$. By analytical continuation of solutions to the elliptic equation (\ref{eqest2a}), $V$ must vanish identically, a contradiction. Thus, $0$ is a regular value of $V$ and hence, by (\ref{eqcomp2}), $R$ must be constant. The remaining assertions follow because the second fundamental form of the level set $\{V=0\}$ and the derivative of $|\nabla V|^2$ on $\{V=0\}$ depends on $Hess V$, which is zero on this set by (\ref{eqest2a}).
\end{proof}
\indent The above proposition implies that static triples $(M^n,g,V)$ as defined in the Introduction have constant scalar curvature and totally geodesic boundary. When dealing with the classification problem, one can therefore assume $(M^n,g,V)$ has scalar curvature $R=\epsilon n(n-1)$ for some $\epsilon\in\{-1,0,1\}$. The constant value of $|\nabla V|$ on a connected component of $\partial M$ is sometimes called the \textit{surface gravity} of the component (as in \cite{GibHaw}). \\
\indent As a consequence of the Maximum Principle, the following result holds for compact static triples. 
\begin{prop}[\cite{FisMar}, \cite{Bour}] \label{propcompactas}
 Let $(M^n,g,V)$ be a static triple with scalar curvature $R=\epsilon n(n-1)$, $\epsilon\in\{-1,0,1\}$.
 \begin{itemize}
 \item[$i)$] If $M$ is closed, then $V$ is a positive constant and $(M,g)$ is Ricci-flat. 
 \item[$ii)$] If $M$ is compact and $\partial M$ is non-empty, then $R>0$ and $n$ is the first Dirichlet eigenvalue of the Laplacian of $(M,g)$.
 \end{itemize}
\end{prop}
\begin{proof} 
If $M$ is closed, Hopf's Maximum Principle implies that $R$ is zero and $V$ is constant. Thus, the static equation (\ref{eqestatica}) becomes $Ric = 0$. If $M$ is compact and $\partial M\neq \emptyset$, $R/(n-1)$ is an eigenvalue of the Laplacian for the Dirichlet problem in $(M,g)$, by equation (\ref{eqest2a}). Hence, it is positive. Since $V$ does not change its sign, $R/(n-1)$ is the first eigenvalue.
\end{proof}
\indent Regarding the classification problem of static triples, item $ii)$ of Proposition \ref{propcompactas} has a useful consequence. In order to verify that two compact static triples with the same scalar curvature are equivalent it is enough to verify that they are isometric because, as first eigenfunctions, the static potentials must be proportional to each other in that case. \\
\indent Given a static triple $(M^n,g,V)$ with scalar curvature $R=\epsilon n(n-1)$, the function $|\nabla V|^2+\epsilon V^2$ plays an important role. If $(M^n,g)$ is Einstein, in which case the static equation reduces to Obata's equation $Hess V = -\epsilon V g$ \cite{Oba}, it is possible to verify that this function is constant. The following proposition shows the converse. More generally, we have
\begin{prop}\label{propmaxprin}
 Let $(M^{n},g,V)$ be a static triple with scalar curvature $R = \epsilon n(n-1)$, $\epsilon\in\{-1,0,1\}$. Then the function
 \begin{equation*}
  |\nabla V|^2+\epsilon V^2
 \end{equation*}
 does not achieve a non-negative maximum in $M\setminus \partial M$ unless it is constant and $(M^n,g)$ is Einstein.
\end{prop}
\begin{proof}
 Using the Bochner formula, one proves Shen's identity \cite{Shen},
 \begin{equation} \label{shenform}
  div \left(\frac{1}{V}d\left(|\nabla V|^2+\epsilon V^2\right)\right) = 2V|\mathring{Ric}|^2,
 \end{equation}
 \noindent and applies Hopf's Maximum Principle.
\end{proof}
\indent Integrating (\ref{shenform}) on a compact three-dimensional static triple, one obtains a fundamental formula that can be used to prove Theorem \ref{thm4pi}.
\begin{prop}[\cite{Shen}] \label{propshenform}
 Let $(M^3,g,V)$ be a compact oriented static triple with scalar curvature $R=6$. Denoting by $\partial_1 M, \ldots, \partial_r M$  the connected components of $\partial M$ and by $k_{i}$ the value of $|\nabla V|$ on $\partial_{i} M$, the following formula holds:
\begin{equation} \label{eqgaussbonnetchern1}
 \sum_{i=1}^{r} k_{i}|\partial_i M| + \int_{M} |\mathring{Ric}|^2 Vd\mu = 2\pi\sum_{i=1}^{r}k_i\chi(\partial_{i} M),
\end{equation}
\noindent where $\chi(\partial_i M)$ denotes the Euler characteristic of $\partial_i M$.
\end{prop}
\begin{proof}
\indent By Shen's identity (\ref{shenform}), the vector field
\begin{equation*}
 X=\frac{1}{V}\nabla(|\nabla V|^2+V^2)
\end{equation*} 
is such that $div X = 2V|\mathring{Ric}|^2$ on $M$. Moreover, as can be seen using the static equation (\ref{eqest2a}), $X = 2(Ric(\nabla V,-) - 2\nabla V)$. The Gauss equation for the totally geodesic boundary $\partial M$ gives
\begin{equation*}
 K = \frac{R}{2} -Ric\left(\frac{\nabla V}{|\nabla V|},\frac{\nabla V}{|\nabla V|}\right) = 3-Ric\left(\frac{\nabla V}{|\nabla V|},\frac{\nabla V}{|\nabla V|}\right),
\end{equation*}
\noindent where $K$ denotes the Gaussian curvature of $\partial M$. Observing that the outward-pointing normal to $\partial M$ is $-\nabla V/|\nabla V|$, we integrate (\ref{shenform}) to obtain
\begin{multline*}
 \int_{M}V|\mathring{Ric}|^2d\mu = \frac{1}{2}\int_{M}div X d\mu =\frac{1}{2}\int_{\partial M} g\left(X,-\frac{\nabla V}{|\nabla V|}\right)d\sigma= \\ =\int_{\partial M} (2- Ric\left(\frac{\nabla V}{|\nabla V|},\frac{\nabla V}{|\nabla V|}\right))|\nabla V|d\sigma = \int_{\partial M} (K-1)|\nabla V| d\sigma.
\end{multline*}
\indent Formula (\ref{eqgaussbonnetchern1}) follows now by the Gauss-Bonnet theorem.
\end{proof}
\indent In order to understand the behaviour of the static potential near $\partial M$, we compute the expansions of $V$ and $|\nabla V|$ along a normalized geodesic issuing from $\partial M$ orthogonally.
\begin{prop} \label{propexpansao}
 Let $(M^n,g,V)$ be a static triple with scalar curvature $R=\epsilon n(n-1)$, $\epsilon\in\{-1,0,1\}$. Given $p\in \partial M$, let $\gamma:[0,\epsilon)\rightarrow M$ be the normalized geodesic such that $\gamma(0)=p$ and $\gamma'(0)$ is the unit normal to $\partial M$ pointing inside $M$. Then the following expansions hold as $s$ goes to zero:
 \begin{align}
  V^2\circ\gamma(s) = |\nabla V|^2(p) (s^2 + \frac{1}{3}(Ric(\gamma'(0),\gamma'(0)) - \epsilon n)s^4 + O(s^{6})), \\
  |\nabla V|^2\circ\gamma(s) = |\nabla V|^{2}(p)( 1 + (Ric(\gamma'(0),\gamma'(0)) - \epsilon n)s^2 + O(s^4)).
 \end{align}
\end{prop}
\begin{proof}
Without loss of generality, we can assume that $|\nabla V|(p)=1$, so that $\gamma'(0)=\nabla V(p)$. We need the formulas obtained by taking successive derivatives of the static equation (\ref{eqest2a}):
 \begin{align*}
  Hess V & = V(Ric - \epsilon ng), \\
  \nabla Hess V   & = (Ric - \epsilon n g)\oplus dV + V\nabla Ric, \\
  \nabla^{2} Hess V & = V(Ric -\epsilon ng)\oplus(Ric -\epsilon ng) + V\nabla^2 Ric + T,
 \end{align*}
 \noindent where $T$ depends on $\nabla Ric$ and $dV$. \\
 \indent For $p\in \partial M$, we claim that $\nabla Ric(\nabla V,\nabla V,\nabla V)(p)=0$. In fact, given an orthonormal basis $\{e_{1},\ldots,e_{n-1},\nabla V(p)\}$ of $T_{p}\partial M$, since $(M,g)$ has constant scalar curvature, the contracted second Bianchi inequality implies
 \begin{equation*}
  0 =  \frac{dR}{2} = div Ric (\nabla V) = \sum_{i=1}^{n-1}(\nabla_{e_i}Ric)(\nabla V,e_{i}) + (\nabla_{\nabla V} Ric)(\nabla V,\nabla V).
 \end{equation*}
 \indent On the other hand, since $\partial M$ is totally geodesic, Codazzi equation implies $Ric(\nabla V,e_{i})=0$ for all $i=1,\ldots,n-1$, and therefore $\sum(\nabla_{e_{i}}Ric)(\nabla V,e_{i}) = 0$. This proves the claim. \\
 \indent Hence, at the point $p\in\partial M$, 
 \begin{align*}
  Hess V(\nabla V,\nabla V)(p) & =  0, \\
  (\nabla Hess V)(\nabla V,\nabla V,\nabla V)(p)   & =  Ric(\nabla V,\nabla V)(p) -\epsilon n, \\
  (\nabla^{2} Hess V)(\nabla V,\nabla V,\nabla V,\nabla V)(p) & =  0.
 \end{align*}
 \indent The result follows now by a direct computation using the above formulas.
\end{proof}
\indent As a consequence of the Propositions \ref{propmaxprin} and \ref{propexpansao}, we obtain some control on the behaviour of the curvature of the boundary component where the function $|\nabla V|^2+\epsilon V^2$ attains a maximum.
\begin{prop} \label{propcurvbordo}
 Let $(M^n,g,V)$ be a static triple with scalar curvature $R^{M}=\epsilon n(n-1)$, $\epsilon\in\{-1,0,1\}$. If $|\nabla V|^2+ \epsilon V^2$ attains a maximum at a boundary component $\partial_1 M$, then $\partial_1 M$ has scalar curvature greater than or equal to $\epsilon (n-1)(n-2)$.
\end{prop}
\begin{proof}
 Since $|\nabla V|$ is constant on each connected component of $\partial M=V^{-1}(0)$, the maximum of $|\nabla V|^2+ \epsilon V^2$ is attained at every point of $\partial_1 M$. Let $\gamma$ be a geodesic issuing from $p\in\partial_1 M$ as in Proposition \ref{propexpansao}. Then $|\nabla V|^2(p) \geq (|\nabla V|^2 + \epsilon V^2)\circ \gamma(s) = |\nabla V|^2(p)(1 + (Ric(\gamma'(0),\gamma'(0)) - \epsilon (n-1))s^2 + O(s^4))$ implies that $Ric(\gamma'(0),\gamma'(0)) \leq \epsilon (n-1)$. Since $\partial_1 M$ is totally geodesic, the result follows from the Gauss equation $R^{\partial_{1} M} = R^{M} - 2Ric(\gamma'(0),\gamma'(0))$.
\end{proof}
\indent We finish this section with a result about the Jacobi operator of a minimal hypersurface $\Sigma$ in $(M^n,g,V)$. We denote by $N$ a unit vector field normal to $\Sigma$ and by $A$ the second fundamental form of $\Sigma$ given by $A=g(\nabla_{X}N,Y)$ for all $X,Y$ tangent to $\Sigma$.
\begin{prop} \label{propjacobi}
 Let $(M^n,g,V)$ be a static triple. Let $\Sigma^{n-1}$ be an immersed compact hypersurface in $M$. Then the Jacobi operator of $\Sigma$ is such that
 \begin{equation*}
  \mathcal{L}_{\Sigma} V := \Delta_{\Sigma} V + (Ric(N,N) + |A|^2)V = -g(\nabla V,\vec{H}) + |A|^2V.
 \end{equation*}
\end{prop}
\begin{proof}
 Since $\Delta_{M} f = \Delta_{\Sigma} f + g(\nabla f, \vec{H}) + Hess f(N,N)$ for every function $f\in C^{\infty}(M)$, the formula for $\mathcal{L}_{\Sigma} V$ follows from the static equation (\ref{eqestatica}).
\end{proof}

\section{The Cotton tensor of a three-dimensional static triple}

\indent We deduce some identities satisfied by the Cotton tensor of three-di\-men\-sio\-nal static triples. They will be used in the proof of the Bochner type formula (\ref{eqBochner}) in the next section. It may be interesting to compare the computations carried out in this section and in the following one with the more general computations related to conformal flatness of \cite{QinYua}. \\ 
\indent Let $(M^3,g)$ be a Riemannian three-manifold. The Riemann curvature tensor is completely determined by the Ricci tensor. In fact,
\begin{equation} \label{eqriem3}
 Rm = \left(Ric - \frac{R}{4}g\right)\odot g = \left(\mathring{Ric} + \frac{R}{12}g\right)\odot g,
\end{equation}
\noindent where $A\odot B$ denotes the Kulkarni-Nomizu product of symmetric two-tensors given in coordinates by $(A\odot B)_{ijkl} = A_{ik}B_{jl} - A_{il}B_{jk} - A_{jk}B_{il} + A_{jl}B_{ik}$ (see \cite{Bes}). \\
\indent The Cotton tensor of $(M^3,g)$ is defined by
\begin{multline} \label{defcott}
  \nonumber C(X,Y,Z) = (\nabla_{Z}Ric)(X,Y) - (\nabla_{Y}Ric)(X,Z) \\ 
                        - \frac{1}{4}(dR(Z)g(X,Y) - dR(Y)g(X,Z))
\end{multline}
\noindent for all $X, Y, Z\in \mathcal{X}(M)$. If $V\in C^{\infty}(M)$ is a static potential in $(M^3,g)$, the Cotton tensor has a particularly simple expression.
\begin{lemm} \label{lemt}
 Assume $(M^3,g)$ admits a static potential $V\in C^{\infty}(M)$. Let
 \begin{equation*} 
  \mathring{Hess V} = Hess V - \frac{\Delta V}{3} g
 \end{equation*}
 \noindent denote the traceless part of the Hessian of $V$ and let $T$ be the three-tensor defined by
 \begin{equation*}
  T(X,Y,Z) = (\nabla_{Z} \mathring{Hess V})(X,Y) - (\nabla_{Y} \mathring{Hess V})(X,Z)
 \end{equation*}
 \noindent for all $X, Y, Z\in \mathcal{X}(M)$. Then
 \begin{multline}
   T(X,Y,Z)  = - (\mathring{Ric}(X,Y)dV(Z) - \mathring{Ric}(X,Z)dV(Y)) \\
                 - (g(X,Y)\mathring{Ric}(Z,\nabla V) - g(X,Z)\mathring{Ric}(Y,\nabla V)), \label{eqt}
 \end{multline}
 \begin{align}
   \nonumber       C(X,Y,Z) & = (\nabla_{Z})(\mathring{Ric})(X,Y) - (\nabla_{Y})(\mathring{Ric})(X,Z) \\
                  & = \frac{1}{V}\left(T(X,Y,Z) - (\mathring{Ric}(X,Y)dV(Z) - \mathring{Ric}(X,Z)dV(Y)) \right). \label{eqcott}
 \end{align}
\end{lemm}

\begin{proof}
 Without loss of generality we assume that $R=6\epsilon$ for $\epsilon \in\{-1,0,1\}$. The static equations (\ref{eqest2a}), (\ref{eqest2b}) can be rewritten in coordinates as 
 \begin{align*}
  (\mathring{Hess V})_{ij} & =  V(\mathring{Ric})_{ij}, \\
  (Hess V)_{ij} = V_{;ij} & = (\mathring{Hess V})_{ij} - \epsilon Vg_{ij}.
 \end{align*} 
 \indent By the Ricci identity for commuting covariant derivatives and (\ref{eqriem3}), we obtain formula (\ref{eqt}):
\begin{align*}
            T_{ijk} & =  (V_{;ijk} + \epsilon V_{k}g_{ij}) - (V_{;ikj} + \epsilon V_{j}g_{kj}) = - R_{pikj}\tensor{V}{_;^p} + \epsilon (g_{ij}V_{;k} - g_{kj}V_{j}) \\
                    & =  - ((\mathring{Ric})_{ij}V_{;k} - (\mathring{Ric})_{ik}V_{;j}) - (g_{ij}(\mathring{Ric})_{kp}\tensor{V}{_;^p} - g_{ik}(\mathring{Ric})_{jp}\tensor{V}{_;^p}).
\end{align*}
 \indent On the other hand, $C_{ijk}= (\mathring{Ric})_{ij;k} - (\mathring{Ric})_{ik;j}$, because the scalar curvature is constant. Formula (\ref{eqcott}) follows since $\mathring{Hess V} =  V\mathring{Ric}$.
\end{proof}
\indent The norm of the Cotton tensor $C$ can be computed in different ways, each way allowing to deduce interesting consequences. For instance, on a regular level set $\Sigma=V^{-1}(c)$, $c>0$, $|C|$ depends on the norm of the traceless part of the second fundamental form of $\Sigma$ and on the gradient of $|\nabla V|^2$ in $\Sigma$ (see \cite{Lin}, Formula 14). That computation is useful to prove the classification results of \cite{Kob} and \cite{Laf}. In the next lemma we show other two formulas that will be used in our applications. 
\begin{lemm} \label{lemnormacott}
 Assume $(M^3,g)$ admits a static potential $V\in C^{\infty}(M)$. Then
 \begin{equation*}
   |C|^2 = \frac{1}{V^2}\left( 8|\mathring{Ric}|^2|\nabla V|^2 - 12|\mathring{Ric}(\nabla V,-)|^2\right) = - \frac{4}{V}C_{ijk}\tensor{(\mathring{Ric})}{^i^j}\tensor{V}{^k}.
 \end{equation*}
 \indent In particular, $C$ vanishes at a point $p\in M$ with $V(p)\neq 0$ if and only if $\nabla V(p) = 0$ or $\mathring{Ric}$ has at most two different eigenvalues at $p$.
\end{lemm}
\begin{proof}
\indent By Lemma \ref{lemt}, $VC_{ijk} = T_{ijk} + U_{ijk}$, where
\begin{align*}
 T_{ijk} & = - ((\mathring{Ric})_{ij}V_{;k} - (\mathring{Ric})_{ik}V_{;j}) - (g_{ij}(\mathring{Ric})_{kp}\tensor{V}{_;^p} - g_{ik}(\mathring{Ric})_{jp}\tensor{V}{_;^p}),  \\
 U_{ijk} & = -((\mathring{Ric})_{ij}V_{;k} - (\mathring{Ric})_{ik}V_{;j}).
\end{align*}
\indent Thus, since
\begin{align}
 |U|^2 & = 2\tensor{(\mathring{Ric})}{_i_j}\tensor{(\mathring{Ric})}{^i^j}V_{;p}\tensor{V}{_;^p} - 2\tensor{(\mathring{Ric})}{_j_p}\tensor{V}{_;^p}\tensor{(\mathring{Ric})}{^j^q}\tensor{V}{_;_q}, \label{eqoutrotermo} \\
 T_{ijk}U^{ijk} & = 2\tensor{(\mathring{Ric})}{_i_j}\tensor{(\mathring{Ric})}{^i^j}V_{;p}\tensor{V}{_;^p} - 4\tensor{(\mathring{Ric})}{_j_p}\tensor{V}{_;^p}\tensor{(\mathring{Ric})}{^j^q}\tensor{V}{_;_q}, \label{eqtermomisto} \\
  |T|^2 & = 2\tensor{(\mathring{Ric})}{_i_j}\tensor{(\mathring{Ric})}{^i^j}V_{;p}\tensor{V}{_;^p} - 2\tensor{(\mathring{Ric})}{_j_p}\tensor{V}{_;^p}\tensor{(\mathring{Ric})}{^j^q}\tensor{V}{_;_q}, \label{eqnormat}
\end{align}
\noindent we obtain
\begin{equation*}
 V^2|C|^2 = |T|^2 + 2T_{ijk}U^{ijk} + |U|^2 = (8|\mathring{Ric}|^2 - 12|\mathring{Ric}(\nabla V,-)|^2).
\end{equation*}
\indent Moreover, since $C_{ijk}+C_{ikj}=0$, from (\ref{eqoutrotermo}) and (\ref{eqtermomisto}) we have
\begin{align*}
 - VC_{ijk}\tensor{(\mathring{Ric})}{^i^j}\tensor{V}{^k} & = \frac{V}{2}C_{ijk}U^{ijk}= \frac{1}{2}(T_{ijk} + U_{ijk})U^{ijk} \\ 
               & = (2|\mathring{Ric}||\nabla V|^2 - 3 |\mathring{Ric}(\nabla V,-)|^2) = \frac{V^2}{4}|C|^2.
\end{align*}
\indent To finish the proof of the lemma, given $p\in M$ with $V(p)\neq 0$, we choose $\{e_1,e_2,e_3\}\subset T_{p}M$ an orthonormal basis diagonalizing $\mathring{Ric}$. Writing $(\mathring{Ric})_{ij}=\lambda_{i}\delta_{ij}$ and $\nabla V(p) = \sum_{i=1}^{3} a_ie_{i}$, we have
\begin{multline*}
 |C|^2=8|\mathring{Ric}|^2(p)|\nabla V|^2(p) - 12|(\mathring{Ric})(\nabla V,-)|^2(p) = \\ = 4a_{1}^2(2\lambda_{2}^2+2\lambda_{3}^2-\lambda_{1}^2) + 4a_{2}^2(2\lambda_{1}^2+2\lambda_{3}^2-\lambda_{2}^2) + 4a_{3}^2(2\lambda_{1}^2+2\lambda_{2}^2-\lambda_{3}^2).
\end{multline*}
\indent Since $tr(\mathring{Ric})=\lambda_{1}+\lambda_{2}+\lambda_{3} = 0$, we obtain
\begin{equation*}
 |C|^2 = 4a_{1}^2(\lambda_{1}+2\lambda_{2})^2 + 4a_{2}^2(\lambda_{2}+2\lambda_{1})^2 + 4a_{3}^2(\lambda_{1}-\lambda_{2})^2.
\end{equation*} 
\indent The last statement of the lemma follows.
\end{proof}

\section{The Bochner type formula}

\begin{prop} Assume $(M^3,g)$ admits a static potential $V\in C^{\infty}(M)$. Then
\begin{equation} \label{eqbochnersecao}
 \frac{1}{2}div(V\nabla|\mathring{Ric}|^2) = (|\mathring{Ric}|^2 + \frac{|C|^2}{2})V + (R |\mathring{Ric}|^2 + 18 det(\mathring{Ric}))V.
\end{equation}
\end{prop}
\begin{proof}
\indent Without loss of generality, we rescale the metric in such way that $R=6\epsilon$ for some $\epsilon\in\{-1,0,1\}$. During the computation, we will use the static equations written as
\begin{align}
  (\mathring{Hess V})_{ij} & =  V(\mathring{Ric})_{ij},  \label{eqesttraco1} \\
   V_{;ij}  & = (\mathring{Hess V})_{ij} - \epsilon Vg_{ij}, \label{eqesttraco2}
\end{align} 
\noindent and the fact that $\mathring{Ric}$ is divergence free, which is a consequence of the contracted second Bianchi identity $2div(Ric) = dR = 0$. \\
\indent We have
\begin{align}
  \nonumber \frac{1}{2}div(V|\mathring{Ric}|^2) & = \frac{1}{2}g(\nabla V,\nabla|\mathring{Ric}|^2) + V\Delta\frac{|\mathring{Ric}|^2}{2} \\
                               & = (\mathring{Ric})_{ij;p}\tensor{(\mathring{Ric})}{^i^j}\tensor{V}{_;^p} + V|\nabla{\mathring{Ric}}|^2 + Vg(\Delta \mathring{Ric},\mathring{Ric}).  \label{eqdiva}                      
\end{align}
\indent By the static equations (\ref{eqesttraco1}) and (\ref{eqesttraco2}),
\begin{align}
\nonumber (\Delta(\mathring{Hess V}))_{ij} & =  (\Delta V)\mathring{Ric}_{ij} + V(\Delta \mathring{Ric})_{ij} + 2(\mathring{Ric})_{ij;p}\tensor{V}{_;^p}  \\ 
  \Rightarrow  V(\Delta\mathring{Ric})_{ij}  & =  (\Delta(\mathring{Hess V}))_{ij} +3\epsilon V(\mathring{Ric})_{ij} - 2 (\mathring{Ric})_{ij;p}\tensor{V}{_;^p}. \label{eqdeltaRic}
\end{align}
\indent Now we compute $\Delta(\mathring{Hess V})$. 
Following the notations of Lemma \ref{lemt}, let $T_{ijk}=(\mathring{Hess})_{ij;k} - (\mathring{Hess})_{ik;j}$. Taking two derivatives of $\mathring{Hess V}$, we obtain 
\begin{align*}
 (\mathring{Hess V})_{ij;kl} & = (\mathring{Hess V})_{ik;jl} + T_{ijk;l} = (\mathring{Hess V})_{ki;jl} + T_{ijk;l} \\
                             & = (\mathring{Hess V})_{ki;lj} + \tensor{R}{^p_k_j_l}(\mathring{Hess V})_{pi} + \tensor{R}{^p_i_j_l}(\mathring{Hess V})_{kp} + T_{ijk;l} \\
                             & = (\mathring{Hess V})_{kl;ij} + \tensor{R}{^k_p_l_j}\tensor{(\mathring{Hess V})}{^p_i} - R_{pilj}\tensor{(\mathring{Hess V})}{_k^p} \\
                             & \quad + T_{kil;j} + T_{ijk;l}.
\end{align*}
\indent Contracting the two last indexes and using (\ref{eqriem3}), we have
\begin{align}
 \nonumber (\Delta(\mathring{Hess V}))_{ij} & = \tensor{(\mathring{Hess V})}{^q_q_;_i_j} + R_{pj}\tensor{(\mathring{Hess V})}{^p_i} - R_{piqj}\tensor{(\mathring{Hess V})}{^q^p} \\
 \nonumber                       & \quad  + \tensor{T}{^q_i_q_;_j} + \tensor{T}{_i_j_q_;^q} \\
 \nonumber                       & =  \tensor{(\mathring{Ric})}{_p_j}\tensor{(\mathring{Hess V})}{^p_i} + 2\epsilon g_{pj}\tensor{(\mathring{Hess V})}{^p_i} \\
 \nonumber                       & \quad - \epsilon (g_{pq}g_{ij}-g_{pj}g_{iq})\tensor{(\mathring{Hess V})}{^q^p} \\ 
  \nonumber                      & \quad - (\mathring{Ric})_{pq}g_{ij}\tensor{(\mathring{Hess V})}{^q^p} + (\mathring{Ric})_{pj}g_{iq}\tensor{(\mathring{Hess V})}{^q^p} \\
  \nonumber                      & \quad  + (\mathring{Ric})_{iq}g_{pj}\tensor{(\mathring{Hess V})}{^q^p} - (\mathring{Ric})_{ij}g_{pq}\tensor{(\mathring{Hess V})}{^q^p} \\
  \nonumber                      & \quad + \tensor{T}{^q_i_q_;_j} + \tensor{T}{_i_j_q_;^q}  \\
  \nonumber                      & =  3\epsilon V(\mathring{Ric}_{ij}) - Vg_{ij}(\mathring{Ric})_{pq}\tensor{(\mathring{Ric})}{^q^p} \\
  \nonumber                      & \quad + 2V(\mathring{Ric})_{pj}\tensor{(\mathring{Ric})}{^p_i} + V(\mathring{Ric})_{iq}\tensor{(\mathring{Ric})}{_j^q} \\
                                 & \quad + \tensor{T}{^q_i_q_;_j} + \tensor{T}{_i_j_q_;^q}. \label{eqHessa}
\end{align} 
\indent By Lemma \ref{lemt},
\begin{equation*}
 T_{ijk} = - ((\mathring{Ric})_{ij}V_{k} - (\mathring{Ric})_{ik}V_{j}) - (g_{ij}(\mathring{Ric})_{kp}\tensor{V}{_;^p} - g_{ik}(\mathring{Ric})_{jp}\tensor{V}{_;^p}).
\end{equation*}
\indent Hence,
\begin{align}
 \nonumber \tensor{T}{^q_i_q_;_j} & = -((\tensor{(\mathring{Ric})}{^q_i}V_{;q} - \tensor{(\mathring{Ric})}{^q_q}V_{;i}) + (\tensor{g}{^q_i}(\mathring{Ric})_{qp}\tensor{V}{_;^p}  - \tensor{g}{^q_q}(\mathring{Ric})_{ip}\tensor{V}{_;^p}))_{;j} \\
 \nonumber                        & = ((\mathring{Ric})_{ip}\tensor{V}{_;^p})_{;j} = (\mathring{Ric})_{ip;j}\tensor{V}{_;^p} + (\mathring{Ric})_{ip}\tensor{V}{_;^p_j} \\
 \nonumber                        & = (\mathring{Ric})_{ip;j}\tensor{V}{_;^p} + (\mathring{Ric})_{ip}(\tensor{(\mathring{Hess V})}{^p_j} - \epsilon V\tensor{g}{^p_j} ) \\
                                  & = (\mathring{Ric})_{ip;j}\tensor{V}{_;^p} + V(\mathring{Ric})_{ip}\tensor{(\mathring{Ric})}{^p_j} - \epsilon V (\mathring{Ric})_{ij}. \label{eqTa}
\end{align}
\indent Moreover, recalling that $\tensor{(\mathring{Ric})}{_i_q_;^q}=0$, we also have
\begin{align}
\nonumber \tensor{T}{_i_j_q_;^q} & =  -( \tensor{(\mathring{Ric})}{_i_j_;^q}V_{;q} - \tensor{(\mathring{Ric})}{_i_q_;^q}\tensor{V}{_;_j}) - ( (\mathring{Ric})_{ij}\tensor{V}{_;_q^q} - (\mathring{Ric})_{iq}\tensor{V}{_;_j^q}) \\
\nonumber                        & \quad - (g_{ij}\tensor{(\mathring{Ric})}{_q_p_;^q} - g_{iq}\tensor{(\mathring{Ric})}{_j_p_;^q})\tensor{V}{_;^p} - (g_{ij}\tensor{(\mathring{Ric})}{_q_p} - g_{iq}\tensor{(\mathring{Ric})}{_j_p})\tensor{V}{_;^p^q} \\
\nonumber                        & = - \tensor{(\mathring{Ric})}{_i_j_;_q}\tensor{V}{_;^q} - (-3\epsilon V)(\mathring{Ric})_{ij} + (\mathring{Ric})_{iq}(\tensor{(\mathring{Hess V})}{^q_j} - \epsilon V \tensor{g}{^q_j}) \\
\nonumber                        & \quad + \tensor{(\mathring{Ric})}{_j_p_;_i}\tensor{V}{_;^p} - (g_{ij}\tensor{(\mathring{Ric})}{_q_p} - g_{iq}\tensor{(\mathring{Ric})}{_j_p})(\tensor{(\mathring{Hess V})}{^p^q} - \epsilon V\tensor{g}{^p^q}) \\\
\nonumber                        & = \epsilon V (\mathring{Ric})_{ij} + V(\mathring{Ric})_{iq}\tensor{(\mathring{Ric})}{^q_j} + V(\mathring{Ric})_{jp}\tensor{(\mathring{Ric})}{^p_i} \\
                                 & \quad    - Vg_{ij}(\mathring{Ric})_{pq}\tensor{(\mathring{Ric})}{^q^p} + (\mathring{Ric})_{jp;i}\tensor{V}{_;^p} - (\mathring{Ric})_{ij;q}\tensor{V}{_;^q}. \label{eqTb}
\end{align}
\indent Putting together (\ref{eqHessa}), (\ref{eqTa}) and (\ref{eqTb}) we conclude that
\begin{align}
 \nonumber (\Delta (\mathring{Hess V}))_{ij} & = 3\epsilon V (\mathring{Ric})_{ij} - 2V(\mathring{Ric})_{pq}(\mathring{Ric})^{pq}g_{ij} \\
   \nonumber                          & \quad  + 3V((\mathring{Ric})_{jp}\tensor{(\mathring{Ric})}{^p_i} + (\mathring{Ric})_{ip}\tensor{(\mathring{Ric})}{^p_j}) \\
                               & \quad + ((\mathring{Ric})_{ip;j}+(\mathring{Ric})_{jp;i} - (\mathring{Ric})_{ij;p})\tensor{V}{_;^p}. \label{eqdeltaHess}
\end{align}
\indent Combining (\ref{eqdiva}), (\ref{eqdeltaRic}) and (\ref{eqdeltaHess}) we finally obtain
\begin{align*}
 \nonumber \frac{1}{2}div(V\nabla|\mathring{Ric}|^2) & = V|\nabla \mathring{Ric}|^2 + 6\epsilon V|\mathring{Ric}|^2  \\ 
 \nonumber       & \quad  + 3V(\tensor{(\mathring{Ric})}{_i_p}\tensor{(\mathring{Ric})}{^p_j}+\tensor{(\mathring{Ric})}{_j_p}\tensor{(\mathring{Ric})}{^p_i})\tensor{(\mathring{Ric})}{^i^j} \\
 \nonumber & \quad - 2((\mathring{Ric})_{ij;p}-(\mathring{Ric})_{ip;j})(\mathring{Ric})^{ij}\tensor{V}{_;^p} \\ 
\nonumber      &  =   V|\nabla \mathring{Ric}|^2 - 2C_{ijp}\tensor{(\mathring{Ric})}{^i^j}\tensor{V}{_;^p}  \\ 
               & \quad + 6\epsilon V|\mathring{Ric}|^2 + 6 Vtr(\mathring{Ric}\circ\mathring{Ric}\circ\mathring{Ric}).
\end{align*}
\indent To finish the proof, we use Lemma \ref{lemnormacott} and observe that if $\lambda_1,\lambda_2,\lambda_3$ are the eigenvalues of $\mathring{Ric}$ at some point of $M$, then $\lambda_1+\lambda_2+\lambda_3=0$ implies 
\begin{equation*}
 tr(\mathring{Ric}\circ\mathring{Ric}\circ\mathring{Ric}) = \lambda_1^3+\lambda_2^3+\lambda_3^3 = - 3\lambda_1^2\lambda_2 - 3\lambda_1\lambda_2^2 = 3\lambda_1\lambda_2\lambda_3 = 3det(\mathring{Ric}).
\end{equation*}
\end{proof}
\section{Proof of Theorem A}

\begin{thm} \label{thmBochner}
 Let $(M^3,g,V)$ be a compact oriented static triple with positive scalar curvature. If $|\mathring{Ric}|^2 \leq R^2/6$, then one of the following alternatives holds:
 \begin{itemize}
  \item[$i)$] $\mathring{Ric}= 0$ and $(M^3,g,V)$ is equivalent to the standard hemisphere; or
  \item[$ii)$] $|\mathring{Ric}|^2 = R^2/6$ and $(M^3,g,V)$ is covered by a static triple that is equivalent to the standard cylinder.
 \end{itemize}
\end{thm}

\begin{proof}
 \indent Using Lagrange's multipliers method, one proves that every symmetric traceless linear endomorphism $A:\mathbb{R}^3 \rightarrow \mathbb{R}^3$ satisfies
 \begin{equation*}
    54(det(A))^2 \leq |A|^6, 
 \end{equation*}
  \noindent with equality if and only if $A$ has at most two distinct eigenvalues. Hence, since $V$ is non-negative, the last two terms of the Bochner type formula (\ref{eqbochnersecao}) satisfy the inequality
\begin{equation} \label{ineqauto}
 (R|\mathring{Ric}|^2 + 18 det(\mathring{Ric}))V \geq (R-\sqrt{6}|\mathring{Ric}|)|\mathring{Ric}|^2V,
\end{equation}
\noindent with equality at $p\in M\setminus\partial M$ if and only if $\mathring{Ric}(p)$ has at most two distinct eigenvalues. Since $V$ vanishes on $\partial M$, integrating (\ref{eqbochnersecao}) on $M$ and using (\ref{ineqauto}) we obtain
 \begin{equation*}
  0 \geq \int_{M} \left(|\nabla\mathring{Ric}|^2 + \frac{|C|^2}{2}\right)Vd\mu + \int_{M} \left(R-\sqrt{6}|\mathring{Ric}|\right)|\mathring{Ric}|^2V d\mu.
 \end{equation*}  
 \indent Our assumptions therefore imply that $(M^3,g,V)$ is a locally conformally flat static triple with parallel Ricci tensor, which either satisfies $\mathring{Ric} = 0$ or $R=\sqrt{6}|\mathring{Ric}|$. The result follows then by the classification results of \cite{Kob} and \cite{Laf}, see Theorem \ref{thmconfflat} in the Introduction. 
\end{proof}

\section{The associated singular Einstein manifold}

\indent Let $(M^n,g,V)$ be a static triple. As usual, we assume it has constant scalar curvature $\epsilon n(n-1)$ for some $\epsilon\in\{\-1,0,1\}$. \\ 
\indent Let $\mathcal{U}$ be a small tubular neighborhood of $\partial M$ diffeomorphic to $[0,1)\times \partial M$. Given $p\in \mathcal{U}$ and the corresponding point $(r,x)\in[0,1)\times \partial M$, we write $p=(r,x)$, for simplicity. Let $B_{1}^{2}$ denote the open unit ball centred at the origin of $\mathbb{R}^2$ and let $S^1=\mathbb{R}/2\pi\mathbb{Z}$ denote the unit circle.\\
\indent Let $\mathcal{N}^{n+1}$ be the quotient set
\begin{equation*}
 (S^1 \times (M\setminus\partial M)) \sqcup (B_{1}^2\times\partial M) /\sim,
\end{equation*}
\noindent where the equivalence relation $\sim$ identifies $(\theta,p)\in (S^{1}\times(\mathcal{U}\setminus\partial M))$ with $(r\cos\theta,r\sin\theta,x)\in(B_{1}^2\setminus\{0\})\times\partial M$ if $p=(r,x)$. \\
\indent $\mathcal{N}^{n+1}$ canonically inherits a structure of smooth manifold. $\{0\}\times \partial M$ is a codimension-two submanifold of $\mathcal{N}$ which we can identify with $\partial M\subset M$. \\
\indent One can intuitively imagine $\mathcal{N}^{n+1}$ as the manifold obtained by rotating $M$ around $\partial M$. \\
\indent The map that associates to each $\theta_0\in S^1$ the diffeomorphism of $\mathcal{N}$ determined  by $(\theta, p) \in S^1\times(M\setminus\partial M) \mapsto (\theta+\theta_0,p)\in S^1\times(M\setminus\partial M)$, $p\in \partial M \mapsto p \in \partial M$ is a smooth action of $S^1$ on $\mathcal{N}$ whose set of fixed points is $\partial M\subset \mathcal{N}$. We denote by $X\in\mathcal{X}(\mathcal{N})$ the vector field generating this action. \\
\indent For every $\theta\in S^1$ the map $\phi_{\theta}: M \rightarrow \mathcal{N}$ defined by $p\in M\setminus\partial M \mapsto (\theta,p)\in S^1\times (M\setminus\partial M)$, $p\in\partial M \mapsto (0,p)\in B_{1}^{2}\times\partial M$ is a smooth embedding. It gives a natural identification between $M$ and the space of orbits of the $S^1$ action on $\mathcal{N}$ . \\
\indent We denote by $\Omega$ the open dense subset $\mathcal{N}\setminus \partial M\subset \mathcal{N}$. $\Omega$ is naturally identified with $S^1\times (M\setminus \partial M)$. Since $V$ is positive on $M\setminus \partial M$, we can define a smooth Riemannian metric $h$ on $\Omega$ by the formula
\begin{equation}
 h = V^2d\theta^2 + g.
\end{equation}
\indent The maps $\phi_{\theta}: (M\setminus{\partial M},g)\rightarrow (\Omega,h)$ are then isometric embeddings. Moreover, the vector field $X$ is orthogonal to every hypersurface $\phi_{\theta}(M\setminus \partial M)$ and such that $|X|\circ \phi_{\theta} = V$ on $M\setminus \partial M$. These conditions uniquely determine the metric $h$. \\
\indent A computation in coordinates shows that the Ricci tensor of $h$ is given by
\begin{align*}
 Ric_{h}(X,X) & =  -\left(\frac{1}{V}\Delta_{g}V\right) h(X,X), \\
 Ric_{h}(X,Z) & =  0, \\
 Ric_{h}(Y,Z) & =  Ric_{g}(Y,Z) - \frac{1}{V}(Hess_{g} V)(Y,Z), 
\end{align*}
\noindent where $Y, Z$ are tangent to $M$. Since $V$ is a static potential and $(M^n,g)$ has scalar curvature $\epsilon n(n-1)$, the metric $h$ is Einstein, 
\begin{equation*}
 Ric_{h}= \epsilon nh.
\end{equation*}
\indent It is also a straightforward computation to verify that $X=\partial_{\theta}$ is a Killing vector field, i.e., that the $S^1$ action described before is an action by isometries.\\
\indent To understand the behaviour of the metric $h$  near a connected component $\partial_{i} M$ of $\partial M$, we introduce Fermi coordinates. Let $s$ be the distance function to $\partial M$ in $(M,g)$ and let $(x^2,x^3)$ be a local coordinate system on a neighborhood $W$ of a point in $\partial_{i} M$. Using $(x^{0}=\theta,x^{1}=s,x^2,x^3)$ as local coordinates in $\mathcal{N}\setminus\partial M$, we can write
\begin{equation} \label{conic}
 h = (ds^2 + k_{i}^{2} s^2 d\theta^2 + \sum_{i,j=2}^{3}g_{ij}(0,x^2,x^3)dx^i dx^j) + E,
\end{equation}
\noindent where $k_i$ denote the constant value of $|\nabla V|$ on $\partial_i M$ and $|E|\leq C s^2$ for some constant $C>0$. Indeed, using the fact that $\partial M$ is totally geodesic and Proposition \ref{propexpansao}, we obtain the expansions
\begin{equation*}
  g_{ij}(s,x^1,x^2) = g_{ij}(0,x^1,x^2) + O(s^2) \quad \text{and} \quad V(s,x^1,x^2) = k_{i}s+O(s^3).
\end{equation*}
\indent We remark that $E_{1A}=E(\partial_s,\partial_{x^{A}})=0$ for all $A=0,1,2,3$. Moreover, the coefficients of $E$ in this coordinate system depend only on the coordinates $x^1=s,x^2,x^3$ and have a Taylor expansion in $s$ as $s$ goes to zero with coefficients that are smooth functions of the coordinates $x^2,x^3$. \\
\indent Equation (\ref{conic}) has the interpretation that near a connected component $\partial_{i} M$ of the singular set $\partial M\subset \mathcal{N}$, the metric $h$ has the structure of cone metric with link $S^{1}(k_i)$ (the circle with length $2\pi k_i$, a number sometimes called the \textit{cone angle}) along the codimension-two smooth submanifold $(\partial_{i} M,g_{\partial M})$, up to a perturbation. Metrics with this singular behaviour has been already described and studied in the literature. We refer the reader to \cite{LiuShe} (see their definition of singular Riemannian metric), \cite{AtiLeB} (where the singularities are called of ``edge-cone" type) and to the more general notion of iterated edge metrics on iterated edge spaces used in \cite{AkuCarMaz1} (see also \cite{AkuCarMaz2}). \\
\indent Summarizing up, the manifold $\mathcal{N}^{n+1}$ endowed with the Einstein edge metric $h$ and the $S^1$ action by isometries generated by the vector field $X$, which is everywhere orthogonal to the orbit space $(M^n,g)$, vanishes on the singular set $\partial M$ and has norm $V$, will be called the \textit{associated singular Einstein manifold} to the static triple $(M^n,g,V)$. \\

\indent We proceed with a description of some topological and geometrical properties of the associated singular Einstein manifold. \\
\indent $\mathcal{N}^{n+1}$ is compact if and only if $M^n$ is compact, and it is simply connected if $M$ and $\partial M$ are simply connected (by Van Kampen's Theorem). When $\mathcal{N}$ is compact, since the vector field $X\in\mathcal{X}(\mathcal{N})$ does not vanish on $\Omega$, $\mathcal{N}$ and $\partial M= \{X=0\}$ have the same Euler characteristic (one can argue as in \cite{Kob2}). When $\mathcal{N}$ is oriented, $X \mapsto -X$ gives rise to an orientation inverting diffeomorphism. \\
\indent The metric $h=V^2d\theta^2+g$ extends smoothly to a component $\partial_{i} M$ of $\partial M\subset \mathcal{N}$ if and only if there exists a smooth function $U : [0,\sqrt{\eta})\times\partial_i M \rightarrow [0,+\infty)$ such that $U(0,p)=1$ for all $p\in \partial M$ and $V(s,p)=sU(s^2,p)$ in the Fermi coordinate system $[0,\eta)\times \partial_{i}M$ (see \cite{Bes}, 9.114). A necessary condition is that $k_{i}=|\nabla V|=1$ on $\partial_i M$. This condition is usually stated to be also sufficient (see \cite{BouGibHor} or \cite{GibHaw}). When $h$ extends smoothly to all $\mathcal{N}$, $(\mathcal{N},h)$ is an example of the Einstein manifolds studied in \cite{Ses}. \\
\indent We emphasize that, in general, equivalent static triples with the same scalar curvature give rise to \textit{geometrically distinct} associated Einstein manifolds. In fact, the orbits of the $S^1$ action will have different lengths if the static potentials are proportional by a factor $\lambda\neq 1$.  \\
\indent When $\partial M$ is compact it is possible normalize the static potential $V$ in such way that $|\nabla V|\leq 1$ on $\partial M$. This means that in this situation, to study $(M^n,g,V)$ one can work with an edge space $(\mathcal{N}^{n+1},h)$ such that all cone angles are less than or equal to $2\pi$. \\
\indent Since $(M,g)$ is complete, $\mathcal{N}$ endowed with the distance function induced by $h$ is a complete length space and the Hopf-Rinow Theorem holds (see \cite{Pfl}, Section 2.4). In particular, $\mathcal{N}$ is compact if and only if $(\mathcal{N},h)$ has finite diameter. \\
\indent The structure of the metric $h$ near the singular set clearly implies that geodesics realizing the distance between a point in $\mathcal{N}$ and a component of $\partial M$ meets $\partial M$ orthogonally. The proof of this fact is essentially the same as the proof of the Gauss' Lemma. \\
\indent  For the computation of the volume of an associated singular Einstein manifold $(\mathcal{N}^{n+1},h)$ to a compact oriented static triple with scalar curvature $R=n(n-1)$, the singular codimension-two set $\partial M$ is negligible. The relation between the volume element of $h$ and the volume element of $g$ is $d\mu_h=Vd\theta\wedge d\mu_g$. Hence, Fubini's theorem and the static equations yield
 \begin{multline*}
  |N|= \int_{0}^{2\pi}\int_{M} d\mu_h = 2\pi\int_{M} V d\mu_g = -\frac{2\pi}{n}\int_{M}\Delta Vd\mu_g = \\ = \frac{2\pi}{n}\int_{\partial M} g\left(\nabla V,\frac{\nabla V}{|\nabla V|}\right)d\sigma_g=\frac{2\pi}{n}\sum_{i=1}^{r}k_i|\partial_i M|.
 \end{multline*}
\indent Each function $\phi$ on $M^n$ corresponds to a function on $\mathcal{N}^{n+1}$ invariant under the $S^1$ action. For the sake of simplicity, we make no notational distinction between them. If $\psi\in C^{\infty}(M)$, a direct computation in coordinates shows that $V\Delta_{h}\psi = V\Delta_{g}\psi + g(\nabla_g V,\psi) = div_{g}(V\nabla_{g}\psi)$ in $\Omega$. \\
\indent In the four-dimensional case, the Riemann curvature tensor of the Einstein metric $h$ is determined by the Weyl tensor $W$ of $h$. An explicit calculus in coordinates shows that $W$ depends only on the Ricci tensor of the metric $g$. In fact, if $\{e^{0}=V^{-1}\partial_{\theta},e^{1},e^2,e^{3}\}$ is an orthonormal basis of $T_{p}\mathcal{N}$, $p\in \Omega$, all the components $W_{ABCD}$ vanish except 
\begin{equation*}
 W_{ijkl} = (\mathring{Ric}_{g} \odot g)_{ijkl}, \quad W_{i0j0} = - (\mathring{Ric}_{g})_{ij} \quad \text{for} \quad i,j,k,l=1,2,3.
\end{equation*}
\indent When $\mathcal{N}^4$ is oriented, considering $W$ as a symmetric endomorphism of the space of two-forms on $M$ and its usual self-dual/anti-self-dual decomposition $W=W^{+}+W^{-}$ (see \cite{Bes}), one can also show by a direct computation that 
\begin{equation*}
 |W^{+}|_{h}^2=|W^{-}|_{h}^2=|\mathring{Ric}_g|^2_{g}.
\end{equation*}
\indent In fact, $W^{+}$ and $-Ric_g$ have the same set of eigenvalues. We will use these facts about $W$ in the last sections. \\
\indent As a final remark concerning the four-dimensional case, we can now unveil the origin of the Bochner type formula (\ref{eqbochnersecao}) in a static triple $(M^3,g,V)$: it is just the Bochner formula for the self-dual part of the Weyl tensor of the Einstein metric $h=V^2d\theta^2+g$ (see \cite{Gur}, Formula 1.3). \\

\indent We finish our description of the associated singular Einstein manifolds to static triples by giving some examples. \\

\noindent $1)$ Given the unit hemisphere $(S^{n}_{+},g_{can},V = x_{n+1})$, the associated Einstein manifold is the unit round sphere $(S^{n+1},g_{can})$. The $S^1$ action is generated by rotations in the plane $\{x_{1}=\ldots=x_{n}=0\}\subset \mathbb{R}^{n+2}$. \\
 
\noindent $2)$ Given the standard three-dimensional cylinder over $S^2$ with scalar curvature $6$, the associated Einstein manifold is the product
 \begin{equation*}
  (S^2\times S^2, \frac{1}{3}(g_{can}+g_{can})).
 \end{equation*}
 \indent The $S^1$ action is generated by rotations of a $S^2$ factor around a fixed axis. \\
 
\noindent $3)$ Any Schwarzschild-de Sitter space of positive mass $m$ has the property that the norm of the gradient of its static potential attains different values on different boundary components. No rescaling of $V$ can remove the singularities of the associated Einstein manifold, which is, topologically, $S^2\times S^2$. \\

\section{Topology of compact static three-manifolds with positive scalar curvature}

\indent The topology of three-dimensional static triples $(M^3,g,V)$ can be studied using area-minimizing surfaces which can be produced in mean-convex manifolds by variational methods if the topology of $M$ is sufficiently complicated. For, as we will see, locally area-minimizing surfaces exist in $M\setminus \partial M$ only in exceptional cases. \\
\indent Let $(M^3,g,V)$ be a compact static triple. Let $\Sigma^2$ be a compact surface and $\phi : \Sigma \rightarrow M$ a proper embedding, i.e., an embedding of $\Sigma$ in $M$ such that $\phi(\Sigma) \cap \partial M = \phi(\partial \Sigma)$. Sometimes we identify $\Sigma$ with its image in $M$. We will frequently assume that $\Sigma$ is two-sided, i.e., that there exists a smooth unit normal vector field $N$ on $\Sigma$. \\
\indent Let $(\overline{M},\overline{g}):=(M\setminus\partial M,V^{-2}g)$. Since $M$ is compact, $V$ vanishes on $\partial M$ and $|\nabla V|$ is a positive constant on each boundary component of $M$, $(\overline{M},\overline{g})$ is a conformally compact Riemannian manifold. These Riemannian manifolds have bounded geometry \cite{AmmLauNis}, i.e., they have positive injectivity radius and the Riemann curvature tensor and all of its covariant derivatives are bounded. This allows one to consider, for some small $\delta>0$, the smooth flow $\Phi:[0,\delta)\times \Sigma \rightarrow M$ of $\Sigma$ with normal velocity $V$,
\begin{equation} \label{flow}
 \frac{d}{dt}\Phi_{t}(p) = V(\Phi_t(p))N_{t}(p) \quad \text{for all} \quad p\in \Sigma, \quad \quad \Phi_{0} = \phi.
\end{equation}
\indent In fact, (\ref{flow}) is just the flow of $\Sigma$ in $(\overline{M},\overline{g})$ by equidistant surfaces, which is smooth if $\delta$ is less than the injectivity radius of $(\overline{M},\overline{g})$. \\
\indent The surfaces $\Phi_{t}(\Sigma)$, $t\in[0,\delta)$, are compact properly embedded surfaces in $(M,g)$ with the same boundary as $\Sigma$, since $V$ vanishes on $\partial M$. \\ 
\indent We show that the flow (\ref{flow}) starting at a minimal surface typically decreases the area (compare with \cite{Bre}, Proposition 3.3, and \cite{LeeNev}, Lemma 2.4). More precisely, we prove the following.
\begin{prop} \label{propflow}
 Let $(M^3,g,V)$ be a static triple with scalar curvature $R=6\epsilon$, $\epsilon\in\{-1,0,1\}$. Let $\Sigma_0^2$ be a compact properly embedded two-sided minimal surface in $(M^3,g)$ and choose $N_0$ a smooth unit normal vector field on $\Sigma_0$. \\
 \indent  Let $\Phi_t$, $t\in[0,\delta)$, be the smooth normal flow with speed $V$ starting at $\Sigma_0$ defined by (\ref{flow}). Write $\Sigma_t=\Phi_{t}(\Sigma_0)$. Then 
 \begin{equation*}
  t\in [0,\delta) \mapsto |\Sigma_t|
 \end{equation*}
 \noindent is monotone non-increasing. When $t\mapsto |\Sigma_t|$ is constant,
 \begin{itemize}
 \item[$i)$] If $\Sigma$ is closed, then 
 \begin{equation*}
  \Phi : ([0,\delta)\times \Sigma_0,(V\circ\Phi)^2dt^2+g_{\Sigma_0}) \rightarrow (M,g)
 \end{equation*}
 \noindent is an isometry onto its image $W\subset M$. Moreover, $(\Sigma_0,g_{\Sigma_0})$ has constant Gaussian curvature $K=3\epsilon$ and $V$ is constant on each $\Sigma_t$, so that $g$ on $W$ is isometric to the product metric $ds^2+g_{\Sigma_0}$. \\
 \item[$ii)$] If $\partial \Sigma$ is not empty, then $\epsilon = 1$ and 
 \begin{equation*}
  \Phi : ([0,\delta)\times (\Sigma_0\setminus \partial\Sigma_0),(V\circ\Phi)^2dt^2+g_{\Sigma_0}) \rightarrow (M,g)
 \end{equation*}
 \noindent is an isometry onto its image $W\subset M\setminus \partial M$. Moreover, $(\Sigma_0,g_{\Sigma_0})$ is isometric to the round hemisphere $(S^2_{+},g_{can})$ with Gaussian curvature $K=1$ and $V$ restricted to $\Sigma_t$ is a static potential, so that $g$ on $W$ is isometric to the constant sectional curvature metric $V_{|\Sigma_0}^2d\theta^2+g_{\Sigma_0}$.
 \end{itemize}
\end{prop}
\begin{proof}
   Let $\vec{H}_t = - H_t N_t$ define the mean curvature $H_t$ of $\Sigma_t$. The variation of $H_t$ is given by $\partial_t H_t = -\mathcal{L}_{\Sigma_t} V$. By Proposition \ref{propjacobi},
 \begin{equation*}
  \partial_{t}H_{t} = g(\nabla V, \vec{H}_t) - |A_t|^2V \leq - g(\nabla V, N_t)H_t.
 \end{equation*}
 \indent Since $H_0=0$, Gronwall's inequality implies that $H_{t} \leq 0$ for every $t\in[0,\delta)$. By the first variation formula of area for variations that fix the boundary $\partial \Sigma_t=\partial \Sigma_{0}$, this implies that $t\in [0,\delta) \mapsto |\Sigma_t|$ is non-increasing. \\
 \indent If $t\in[0,\delta)\mapsto|\Sigma_t|$ is constant, the previous analysis implies that each $\Sigma_t$ is totally geodesic. It follows that each $(\Sigma_t,g_{\Sigma_{t}})$ is isometric to $(\Sigma_0,g_{\Sigma_0})$. \\
 \indent Denote by $V_t$ the restriction of $V$ to $\Sigma_t$. Since each $\Sigma_t$ is totally geodesic, using the general formula for the variation of the second fundamental form of a surface $\Sigma$ under a normal variation with speed $\psi$,
 \begin{equation*}
  \partial_{t} A_{ij} = -\psi R_{iNjN}-(Hess_{\Sigma} \psi)_{ij} + \psi|A|^2,
 \end{equation*} 
 \noindent we deduce that on each $\Sigma_t$ we have
 \begin{equation*}
  Hess_{\Sigma_t} V_t (X,Y) = - g(R(X,N_t)N_t,Y) V_t
 \end{equation*}
 for every pair of vectors $X,Y\in T\Sigma_t$. Given an orthonormal basis $\{X,Y,N_t\}$ of tangent vectors at a point in the totally geodesic surface $\Sigma_t$, we have
 \begin{align*}
  \frac{1}{V_{t}}Hess_{\Sigma_t} V_t (X,X) = \frac{1}{V_t} Hess V_t(X,X)& = Ric(X,X) -3\epsilon \\
                                   & = sec(X\wedge Y) + sec(X\wedge N_t) - 3\epsilon \\
                                   & = K_{t} + g(R(X,N_t)N_t,X) - 3\epsilon \\
                                   & = (K_{t} -3\epsilon) - \frac{1}{V}Hess V_t(X,X).
 \end{align*}
 \indent Since $X\in T\Sigma_t$ is arbitrary, we conclude that $V_t$ is a non-negative smooth function on $\Sigma_t$ that satisfies
 \begin{equation} \label{eqaux}
  Hess_{\Sigma} V_t = \frac{1}{2}(K_t - 3\epsilon)V_tg \quad \text{and} \quad \text{$V_t=0$ on $\partial\Sigma_t$}.
 \end{equation} 
 \indent If $\partial \Sigma_t$ is empty, the existence of a positive solution to (\ref{eqaux}) implies $K_{t}=3\epsilon$ and $V_t$ constant. If $\partial \Sigma_t$ is non-empty, one must have $K=\epsilon =1$ and $V_t$ is a solution to the static equation on $(\Sigma_t,g_{\Sigma_t})$. See Appendix B for a proof of these facts. The last statements of $i)$ and $ii)$ follows by  a reparametrisation. 
\end{proof}

\indent The next proposition exemplifies how one can combine Proposition \ref{propflow} and the existence of area-minimizing surfaces to study the topology of a three-dimensional static triple in a quite effective way.  

\begin{prop} \label{propplateau}
 Let $(M^3,g,V)$ be a compact static triple with non-empty boundary. The inclusion map $i:\partial M \rightarrow M$ induces an injective map $i_{*}: \pi_{1}(\partial M) \rightarrow \pi_{1}(M)$ between the fundamental groups.
\end{prop}
\begin{proof}
 Let $[\Gamma]\in \pi_{1}(\partial M)$ be such that $i_{*}[\Gamma] = 0$ in $\pi_{1}(M)$. We may assume that $\Gamma$ is a smooth embedded closed curve. We want to show that $\Gamma$ bounds a disk in $\partial M$. \\
 \indent Since $\Gamma$ bounds a disk in $M$ and $(M^3,g)$ is a compact mean convex three-manifold, there exists a solution $\Sigma_0$ to the Plateau problem for this given $\Gamma$ \cite{MeeYau}, i.e., $\Sigma_0$ is an disk in $M$ with $\partial \Sigma = \Gamma$ such that 
 \begin{equation*}
  |\Sigma_{0}| = \inf \{|\Sigma|;\, \text{$\Sigma$ immersed disk in $(M^3,g)$ with $\partial \Sigma = \Gamma$}  \}.
 \end{equation*}
 \indent In \cite{MeeYau} it is also shown that any solution is an embedded minimal disk $\Sigma$, with boundary $\partial \Sigma =\Gamma$, which is either contained in $\partial M$ or properly embedded in $M$. In the first case $[\Gamma]=0$ in $\pi_{1}(\partial M)$. Therefore it remains to be analysed the case where $\Sigma_0$ is a properly embedded disk in $M$. \\
 \indent Since $V$ vanishes on $\Gamma \subset \partial M$, the flow (\ref{flow}), $\{\Sigma_t\}_{t\in [0,\delta)}$, with normal speed $V$ starting at $\Sigma_0$, is such that $\partial \Sigma_t = \Gamma$ for all $t\in[0,\delta)$. Thus, $|\Sigma_t|\geq |\Sigma_0|$, as $\Sigma_0$ is a solution to the Plateau problem. By Proposition \ref{propflow}, the opposite inequality also holds. Thus, each $\Sigma_t$ is also a solution to the minimization problem we are considering and the statement $ii)$ in Proposition \ref{propflow} holds true. \\
 \indent To finish the proof, one uses a continuation argument. Let $T>0$ be the maximal time in which the flow (\ref{flow}) exists and is smooth. By the above argument, we are in the situation described by Proposition \ref{propflow}, item $ii)$. We claim that $T=+\infty$. In fact, the surfaces $\Sigma_t\setminus\partial \Sigma_t$ never touch the boundary of $M$ in finite time (for it would imply that $(M\setminus \partial M,V^{-2}g)$ is incomplete) and, if $T<+\infty$, a sequence $\Sigma_{t_{i}}$ with $t_{i}\rightarrow T $ would be a sequence of Plateau solutions converging to another solution to the Plateau problem for $\Gamma$ (which is smooth and properly embedded by \cite{MeeYau}) and it would be possible to continue the flow beyond $T$.  \\ 
 \indent Now we do the same argument flowing in the direction of the opposite normal. In the end, we have obtained an isometric embedding of $(S^{3}_{+}\setminus\partial S^{3}_{+},g_{can})$ in $(M\setminus \partial M,g)$. In particular, $M$ is diffeomorphic to $S^3_{+}$. And in this case, obviously $[\Gamma]=0$ in $\pi_{1}(\partial M)$, as we wanted to prove. 
\end{proof}
\begin{rmk}
 There are oriented static triples $(M^3,g,V)$ with negative scalar curvature and non-spherical boundary components, see \cite{LeeNev}. It is interesting to observe that essentially the same proof as above gives a similar result for the non-compact static triples considered in that paper.
\end{rmk}
\indent We have now all ingredients to prove Theorem B.
\begin{thm}
 Let $(M^3,g,V)$ be a compact oriented static triple with positive scalar curvature. Then
 \begin{itemize}
  \item[$i)$] The universal cover of $M$ is compact.
  \item[$ii)$] If $\partial M$ contains an unstable component, then $\partial M$ contains exactly one unstable component. In this case, $M$ is simply connected.
  \item[$iii)$] Each connected component of $\partial M$ is diffeormorphic to a sphere.
  \end{itemize}
\end{thm}
\begin{proof}
 \indent As usual, we may assume $R=6$. \\
 \indent $i)$ Since $M$ is compact, $|\nabla V| + V^2$ attains its maximum, necessarily at a boundary component $\partial_1 M$ (Proposition \ref{propmaxprin}). By Proposition \ref{propcurvbordo}, $\partial_1 M$ has Gaussian curvature $K\geq 1$ and therefore is diffeomorphic to a two-sphere. This implies that the boundary of the universal covering $\tilde{M}$ has also a component $\partial_1 \tilde{M}$ diffeomorphic to a two-sphere (in fact, every component covering $\partial_{1} M$ satisfies this property). Now, any point in the regular set of the associated singular Einstein manifold $(\mathcal{N}^4,h)$ can be joined to the compact set $\partial_{1}\tilde{M}$ by a geodesic minimizing the distance of $p$ to $\partial_{1} M$ and touching this hypersurface orthogonally. Estimating the distance to the first focal point as in the classical proof of Bonnet-Myers theorem, we conclude by the generalization of the Hopf-Rinow theorem to the singular space $(\tilde{\mathcal{N}^{4}},h)$ (\cite{Pfl}, section 2.4) that $\tilde{\mathcal{N}}$ must be compact, as any of its points must be at distance at most $\pi/2$ from the compact set $\partial_{1}\tilde{M}$.  \\
 \indent $ii)$ We apply the results of the Geometric Measure Theory \cite{FedFle} about the existence of minimizers of area inside an integral homology class. By $i)$, without loss of generality we can assume $M$ is simply connected and show that $\partial M$ has exactly one unstable component. $(M,g)$ cannot be the standard cylinder because it contains by hypothesis at least one unstable boundary component, $\partial_{u} M$. We remark also that all embedded closed surfaces in $M$ are two-sided (as $M$ is simply-connected) and that we can assume $[\partial_u M]\neq 0$ in $H^{2}(M;\mathbb{Z})$ (otherwise, $\partial M=\partial_{u} M$). Since the standard cylinder is ruled out, arguing as in Proposition \ref{propplateau} using \cite{FedFle} and Proposition \ref{propflow}, item $i)$, one shows that the minimization of area in the homology class of $\partial_u M$ produces a smooth embedded oriented stable minimal surface (possibly disconnected and with multiplicities) homologous to $\partial_u M$ that must be contained inside $\partial M$. Clearly, no connected component of such surface is contained in an unstable component of $\partial M$. Reasoning on the homology relations between the boundary components as in \cite{LeeNev}, Lemma 3.3, the existence of an unstable component other than $\partial_u M$ leads to a contradiction. \\
 \indent $iii)$ Each stable component of $\partial M$ is diffeomorphic to a sphere because $M$ is oriented and has positive scalar curvature \cite{SchYau}. On the other hand, if there is an unstable component of $\partial M$, $M$ must be simply connected by $ii)$. The conclusion follows from Proposition \ref{propplateau}. 
\end{proof}

%\begin{rmk} \label{rmkE}
% Let $(M^3,g,V)$ be a compact simply connected static triple with scalar curvature $6$. Any locally area-minimizing boundary component is a two-sphere with area at most $4\pi/3$, and equality implies a local splitting as a product metric \cite{BraBreNev}. Combining this result and the previous arguments, it is possible to give an alternative proof of Theorem E. 
%\end{rmk}

\indent To conclude this section, we describe the two static triples that can be obtained from the standard cylinder. They are the only static quotients of the locally conformally flat examples described in Theorem \ref{thmconfflat}. \\
\indent Given
\begin{equation*}
 (M^3,g,V) = \left( \left[0,\frac{\pi}{\sqrt{3}}\right] \times S^2 , g_{prod}= dt^2 + \frac{1}{3} g_{can}, \frac{1}{\sqrt{3}} \sin(\sqrt{3} t) \right),
\end{equation*}
\noindent the maps 
\begin{equation*}
 A_1 : (t,x)\in M \mapsto (t,-x)\in M, \quad A_2 : (t,x)\in M \mapsto \left(\frac{\pi}{\sqrt{3}} - t,-x\right)\in M
\end{equation*}
\noindent are isometric involutions without fixed points such that $V\circ A_{i}=V$, $i=1,2$. Hence, the quotients $(M^3,g,V)/A_{i}$, $i=1,2$, are static triples. \\
\indent $A_1$ inverts orientation and the first quotient $M/A_{1}$ is diffeomorphic to $[0,\pi/\sqrt{3}]\times \mathbb{RP}^2$. On the other hand, $M/A_{2}$ is oriented, contains an embedded copy of $\mathbb{RP}^2$ and has a \textit{connected} boundary diffeomorphic to $S^2$ (in fact, $M/A_{2}$ is diffeomorphic to $\mathbb{RP}^3$ minus a ball). The associated Einstein manifolds are both isometric to the standard product 
\begin{equation*}
 (S^2\times \mathbb{RP}^2,\frac{1}{3}(g_{can}+ g_{can})),
\end{equation*}
\noindent but the corresponding $S^1$ actions are generated by rotations around a fixed axis of different factors of $S^2\times \mathbb{RP}^2$: $S^2$ in the case of $(M,g,V)/A_1$, $\mathbb{RP}^2$ in the case of $(M,g,V)/A_2$. Remark the subtle fact that, despite $M/A_2$ being orientable, the associated Einstein manifold to $M/A_2$ is non-orientable.

\section{Solution to a Yamabe type problem}

\indent Given $(M^3,g,V)$ a compact static triple with scalar curvature $6$, we consider $(\mathcal{N}^4,h)$ the associated singular Einstein manifold (as described in Section 6). As before, we denote by $\Omega$ the open dense subset of $\mathcal{N}$ where the Riemannian metric $h$ is defined, i.e., $\Omega = \mathcal{N}\setminus \partial M$.  \\ 
\indent The metric $h$ is Einstein with scalar curvature $12$. Let $|W^{+}|_{h}$ be the norm of self-dual part of the Weyl curvature tensor of $h$. As remarked before, $|W^{+}|_h=|\mathring{Ric}_{g}|_{g}$. \\
\indent We consider the equation
\begin{equation} \label{yamabetypeproblem}
 -\Delta_{h} u + \frac{1}{6}(12 - 2\sqrt{6}|W^{+}|_h)u = \frac{1}{6}\lambda u^3,
\end{equation}
\noindent where $\lambda$ is some constant. As showed in \cite{Gur}, the right hand side defines a second order differential operator that has some conformal covariance properties. To explain the geometric relevance of (\ref{yamabetypeproblem}), we observe that a positive function $u\in C^{2}(\Omega)$ that solves (\ref{yamabetypeproblem}) gives rise to a conformal metric $\tilde{h}=u^2h$ in $\Omega$ such that 
\begin{equation*}
 \tilde{R} - 2\sqrt{6}|\tilde{W}^{+}|_{\tilde{h}} = \lambda
\end{equation*}
\noindent is constant (see \cite{Gur}).  \\
\indent The aim of this section is to prove that, except in the case where $W^{+}=0$, there exists for some \textit{non-positive} $\lambda$ a mildly regular solution to (\ref{yamabetypeproblem}) in $(\mathcal{N}^4,h)$ that is invariant under the $S^1$ action on $\mathcal{N}$.
\begin{thm} \label{thmsolveyamabe}
 Let $(\mathcal{N}^4,h)$ be the associated singular Einstein manifold to a compact static triple with positive scalar curvature $(M^3,g,V)$ that is not equivalent to the standard hemisphere. \\
 \indent There exists a constant $\lambda\leq 0$ and a positive function $\phi\in W^{1,2}(\mathcal{N})\cap L^{\infty}(\mathcal{N})$, invariant under the $S^1$ action on $\mathcal{N}$, that is a weak solution to (\ref{yamabetypeproblem}). Moreover, $\phi\in C^{0,\mu}(\mathcal{N})\cap C^{2,\alpha}(\Omega)$ for some $\mu,\alpha\in(0,1)$ and $\phi$ is a strong solution to (\ref{yamabetypeproblem}) in $\Omega$.
\end{thm}
\indent We remark that the Sobolev space $W^{1,2}(\mathcal{N})$ is the completion of $C_{0}^{\infty}(\Omega)$ with respect to the usual Sobolev norm. \\ 
\indent In order to deal with this modified Yamabe problem in the singular space $(\mathcal{N}^4,h)$, we apply the machinery developed by Akutagawa, Carron and Mazzeo in \cite{AkuCarMaz1} to solve the Yamabe problem in singular spaces. In fact, once we reformulate the problem considering the required $S^1$-invariance of the solution, the proof of Theorem \ref{thmsolveyamabe} is essentially the same proof of Theorems 1.12 and 1.15 in \cite{AkuCarMaz1}, combined with the regularity results of \cite{AkuCarMaz2}. \\

\indent As in the study of the classical Yamabe problem, one uses the variational approach. Consider the functional
\begin{equation*} 
 \mathcal{W}(\phi) = \frac{\int_{\mathcal{N}} |\nabla^h \phi|_{h}^2 + \frac{1}{6}(12-2\sqrt{6}|W^{+}|)\phi^2 d\mu_h}{(\int_{\mathcal{N}} \phi^{4}d\mu_h)^{1/2}}
\end{equation*}
\noindent defined for every function $\phi\in W^{1,2}(\mathcal{N})$, $\phi \neq 0$. This functional is well-defined because the usual Sobolev embedding $W^{1,2}(\mathcal{N}) \subset L^{4}(\mathcal{N})$ holds in $(\mathcal{N}^4,h)$ (see \cite{AkuCarMaz1}, section 2.2). It is straightforward to verify that $\mathcal{W}$ is bounded from bellow. \\
\indent Let $A$ denote the set consisting of functions in $W^{1,2}(\mathcal{N})$ that are invariant under the $S^1$ action on $(\mathcal{N}^4,h)$. If $\phi \in A$ is not identically zero, we have
\begin{equation*}
 \mathcal{W}(\phi) = (2\pi)^{1/2}\frac{\int_{M} (|\nabla^{g} \phi|_{g}^2 + \frac{1}{6}(12-2\sqrt{6}|\mathring{Ric}_{g}|_{g})\phi^2 )Vd\mu_g}{(\int_{M} \phi^{4} Vd\mu_g)^{1/2}}.
\end{equation*}
\indent Now we define
\begin{equation*} 
 \mathcal{W}^{M}(\mathcal{N},[h]) := \inf \{\mathcal{W}(\phi);\, \phi\in A, \phi\geq 0, \phi\neq 0\}.
\end{equation*}
\indent A non-negative function $\phi\in A$ attaining the infimum $\mathcal{W}^{M}(\mathcal{N},[h])$ has to be a weak solution of the following equation in $(M^3,g)$:
\begin{equation} \label{yamabetypeproblem1}
 -div_{g}(V\nabla^{g}\phi) + \frac{1}{6}(12-2\sqrt{6}|\mathring{Ric}_{g}|_{g})V\phi = \frac{1}{6} \mathcal{W}^{M}(\mathcal{N},[h])V\phi^3
\end{equation}
\indent This is precisely equation (\ref{yamabetypeproblem}) for the $S^1$-invariant function $\phi\in A$, where $\lambda=\mathcal{W}^{M}(\mathcal{N},[h])$. It is a quasi-linear degenerated elliptic equation (see Appendix C for more information about this type of degenerated equation).   \\
\indent We first argue that the Yamabe type invariant $\mathcal{W}^M(\mathcal{N},[h])$ is \textit{non-positive}. 
The next proposition contains the key inequality, which should be compared with the refined Kato inequality for Codazzi tensors \cite{Bour2}.
\begin{prop} \label{propkato}
 Let $(M^3,g,V)$ be a static triple. At a point $p\in M$ where $\mathring{Ric}$ does not vanishes,
 \begin{equation} \label{deskatoref}
  |\nabla |\mathring{Ric}||^2 \leq  \frac{3}{5}\left( |\nabla \mathring{Ric}|^2 + \frac{|C|^2}{2}\right).
 \end{equation}
\end{prop}
\indent The proof is postponed to the Appendix A, where it is shown more generally that a similar inequality holds true for every symmetric two-tensor that has constant trace and zero divergence. \\
\indent Following \cite{Gur}, Proposition 3.4, the idea to prove that $\mathcal{W}^{M}(N,[h])$ is non-positive is to manipulate the Bochner type formula (\ref{eqbochnersecao}) to get rid of the gradient terms using the refined inequality (\ref{deskatoref}). To a systematic recent use of this idea, we refer the reader to the beautiful paper \cite{BouCar}.
\begin{prop} \label{propyamabeconst}
 Let $(M^3,g,V)$ be a compact oriented static triple with scalar curvature $R=6$. If $\mathring{Ric}\neq 0$ on $M$, then  
 \begin{equation*}
  \inf_{\{\phi\in C^{\infty}(M);\,\phi > 0\}} \int_{M} \left(6|\nabla \phi|^2  + (12 - 2\sqrt{6}|\mathring{Ric}|)\phi^2 \right)Vd\mu_g \leq 0.
 \end{equation*}
\end{prop}
\begin{proof}
 For each $\epsilon > 0$, we consider a suitable power of the positive smooth function $\phi_{\epsilon} = (|\mathring{Ric}|^2+\epsilon)^{1/2}$ and take the limit as $\epsilon$ goes to zero. \\
 \indent By hypothesis, $\phi_\epsilon$ is not constant. Given $p>0$,
\begin{align}
  \nonumber  div(V \nabla(\phi^p_{\epsilon})) & = pdiv(\phi_{\epsilon}^{p-1}(V\nabla \phi_{\epsilon})) \\
  \nonumber                                   & = p\phi_{\epsilon}^{p-1}div(V \nabla \phi_{\epsilon}) + p(p-1)V\phi_{\epsilon}^{p-2}|\nabla \phi_{\epsilon}|^2  \\
  \nonumber                        & = p\phi_{\epsilon}^{p-2}(\phi_{\epsilon}div(V \nabla \phi_{\epsilon}) + (p-1)V|\nabla \phi_{\epsilon}|^2) \\
                                 & = p\phi_{\epsilon}^{p-2}\left(div\left(\frac{V\nabla(\phi_{\epsilon}^2)}{2}\right) + (p-2)V|\nabla \phi_{\epsilon}|^2 \right). \label{eqlaplacep}
\end{align}
\indent Since $\nabla(|\mathring{Ric}|^2) = \nabla (\phi_{\epsilon}^2)=2\phi_{\epsilon}\nabla {\phi_{\epsilon}} $, at a point where $|\mathring{Ric}|\neq 0$ we have
\begin{equation*}
 |\nabla \phi_{\epsilon}|^2 = \frac{1}{4\phi_{\epsilon}^2} |\nabla |\mathring{Ric}|^2|^2 = \frac{|\mathring{Ric}|^2}{|\mathring{Ric}|^2+\epsilon}|\nabla |\mathring{Ric}||^2 \leq |\nabla |\mathring{Ric}||^2.
\end{equation*} 
\indent By Proposition \ref{propkato}, we conclude that
\begin{equation} \label{desaux}
 |\nabla \phi_{\epsilon}|^2 \leq \frac{3}{5}\left(|\nabla \mathring{Ric}|^2 +\frac{|C|^2}{2}\right).
\end{equation}
\indent Inequality (\ref{desaux}) holds even at a point where $\mathring{Ric}=0$, since at such point $\phi_{\epsilon}$ attains a minimum. Combining (\ref{eqlaplacep}), the Bochner type formula (\ref{eqbochnersecao}) and inequality (\ref{desaux}), we obtain
\begin{equation*}
 div(V \nabla(\phi^p_{\epsilon})) \geq p\phi_{\epsilon}^{p-2}V\left(6|\mathring{Ric}|^2 + 18det((\mathring{Ric})^3 + \left(\frac{5}{3} + p - 2\right)|\nabla\phi_{\epsilon}|^2\right). 
\end{equation*}
 \indent Choosing $p=1/3$, the last term vanishes. Recalling inequality (\ref{ineqauto}), we finally obtain
\begin{equation}
 div(V \nabla(\phi^{1/3}_{\epsilon})) \geq \frac{1}{3}V(6 - \sqrt{6}|\mathring{Ric}|)|\mathring{Ric}|^2\phi_{\epsilon}^{-5/3}. \label{desaux2}
\end{equation}
 \indent Multiplying (\ref{desaux2}) by the $\phi_{\epsilon}^{1/3}\in C^{\infty}(M)$ and integrating by parts, we have
 \begin{equation*} 
  \int_{M} (3|\nabla (\phi_{\epsilon}^{1/3})|^2 + (6 - \sqrt{6}|\mathring{Ric}|)|\mathring{Ric}|^2\phi_{\epsilon}^{-4/3}) Vd\mu_g  \leq 0.
 \end{equation*}
 \indent Hence, for all $\epsilon > 0$,
 \begin{multline*}
  \int_{M} (3|\nabla (\phi_{\epsilon}^{1/3})|^2 + (6 - \sqrt{6}|\mathring{Ric}|)(\phi_{\epsilon}^{1/3})^2 Vd\mu_g  \\ \leq \int_{M} (6 - \sqrt{6}|\mathring{Ric}|)\phi_{\epsilon}^{-4/3}(\phi_{\epsilon}^2-|\mathring{Ric}|^2) Vd\mu_g\leq C \epsilon^{1/3}
 \end{multline*}
 \noindent for some constant $C>0$. The proposition follows.
\end{proof}

\indent Proposition \ref{propyamabeconst} implies that 
\begin{equation*}
 \mathcal{W}^{M}(\mathcal{N},[h]) \leq 0
\end{equation*}
\noindent when the compact static triple $(M^3,g,V)$ is not equivalent to the standard hemisphere. \\
\indent Now, following \cite{AkuCarMaz1}, for each open set $\mathcal U\subset \mathcal{N}$ we set
\begin{align*}
 \mathcal{S}(\mathcal{U}) & = \inf\{\int_{\mathcal{N}} |\nabla^{h}\phi|_{h}^2 d\mu_{h};\, \phi\in W_{0}^{1,2}(\mathcal{U}\cap\Omega), \|\phi\|^2_{4}=1\}, \\
 \mathcal{W}(\mathcal{U}) &= \inf\{\mathcal{W}(\phi);\, \phi\in W_{0}^{1,2}(\mathcal{U}\cap\Omega), \|\phi\|^2_{4}=1\},
\end{align*}
and define the local Sobolev and the local (modified) Yamabe constant by 
\begin{equation*}
  \mathcal{S}_{\ell}(\mathcal{N},h) = \inf_{p\in\mathcal{N}} \lim_{r\rightarrow 0} \mathcal{S}(B(p,r)) \quad \text{and} \quad  \mathcal{W}_{\ell}(\mathcal{N},[h])= \inf_{p\in \mathcal{N}} \lim_{r\rightarrow 0}\mathcal{W}(B(p,r)).
\end{equation*}
\indent Accordingly to \cite{AkuCarMaz1},
\begin{equation*}
 \mathcal{S}_{\ell}(\mathcal{N},h)=\mathcal{W}_{\ell}(\mathcal{N},[h])>0,
\end{equation*} 
 \noindent see Lemma 1.3, Proposition 1.4 and Theorem 2.2 in \cite{AkuCarMaz1}. The explicit value of the local Sobolev constant of an edge space like $(\mathcal{N}^4,h)$ was computed in \cite{Mon}. At this point, it may be useful to remark that $R_{h}-2\sqrt{6}|W^{+}|_{h}$ is a bounded function on $\mathcal{N}$ (in fact, it is a Lipschitz function) so that all conditions $iv)$  $a)$, $b)$, $c)$ described in \cite{AkuCarMaz1} for $Scal_g$ (their notation) are satisfied by $R_{h}-2\sqrt{6}|W^{+}|_{h}$. It is straightforward to verify that their arguments apply for such modified potential instead of $Scal_g$. \\
 
\indent With this information at hands, we now sketch the proof of Theorem \ref{thmsolveyamabe} (see \cite{AkuCarMaz1}, proof of Theorem 1.12). The idea to obtain a solution to (\ref{yamabetypeproblem1}) is first to produce minimizers $\phi_{p}\in A$ of the modified functional 
 \begin{equation*}
  \mu_{p}(\phi) = \frac{\int_{\mathcal{N}} |\nabla^h \phi|_{h}^2 + \frac{1}{6}(12-2\sqrt{6}|W^{+}|)\phi^2 d\mu_h}{(\int_{\mathcal{N}} \phi^p d\mu_{h} )^{2/p}}
 \end{equation*}
\noindent in the space of non-negative functions $\phi\in A$ such that $\|\phi\|_{L^{2p/(p-2)}(\mathcal{N})}=1$ by the direct method of the calculus of variations and then, using a priori estimates,  obtain a sequence ${\phi_{p_i}}$ with $p_i\rightarrow 4$ converging to a function $\phi\in A$ such that $\mu_{4}(\phi) = \mathcal{W}^{M}(\mathcal{N},[h])$. \\
\indent In \cite{AkuCarMaz1} it was shown that this approach works if the local Yamabe invariant is \textit{strictly greater} that $\mathcal{W}^{M}(\mathcal{N},h)$. Since we proved $\mathcal{W}^{M}(\mathcal{N},h)$ $\leq 0$, the proof of Theorem 1.12 in \cite{AkuCarMaz1} applies. \\
\indent The conclusion is that there exists a weak solution $\phi\in W^{1,2}(\mathcal{N})\cap L^{\infty}(\mathcal{N})$ to (\ref{yamabetypeproblem1}) that is non-negative and invariant under the $S^1$-action. The positiveness of  $\phi$ follows by \cite{AkuCarMaz1}, Proposition 1.15. Concerning the regularity of the solution, the $C^{2,\alpha}$ regularity on $\Omega$ follows by standard results for quasi-linear elliptic equations \cite{GilTru} and the H\"older continuity on $\mathcal{N}$ follows by \cite{AkuCarMaz2}. \\

\indent We conclude this section with a remark concerning the Yamabe problem on $(\mathcal{N}^4,h)$. We claim that the Yamabe constant of $(\mathcal{N}^4,h)$,
\begin{equation*}
 \mathcal{Y}(\mathcal{N},[h]) = \inf_{\{\psi\in W^{1,2}(\mathcal{N});\,\psi\neq 0\}} \frac{\int_{\mathcal{N}} |\nabla^{h}\psi|^2 + 2\psi^2d\mu_{h}}{(\int_{\mathcal{N}} \psi^4d\mu_{h})^{1/2}},
\end{equation*}
\noindent is positive. In fact, suppose by contradiction that $\mathcal{Y}(\mathcal{N},[h])\leq 0$. Since the local Sobolev constant of $(\mathcal{N}^4,h)$ is positive, Theorems 1.12 and 1.15 of \cite{AkuCarMaz1} implies the existence of a positive function $\psi_{0}\in L^{\infty}(\mathcal{N})\cap W^{1,2}(M)$ that satisfies $-\Delta_{h}\psi_{0}+2\psi_{0} = \mathcal{Y}(\mathcal{N},[h])\psi_{0}^3$ weakly in $\mathcal{N}$. Thus, $\psi_0>0$ satisfies an equation of the form $-\Delta_{h}\psi_{0} + G\psi_{0}=0$ for some positive function $G$ and we have $\int_{\mathcal{N}} |\nabla^{h}\psi_{0}|_{h}^2d\mu_{h}=-\int_{\mathcal{N}} G\psi_{0}^2d\mu_{h} < 0$, a contradiction. \\
\indent In particular, the solution $\phi$ obtained in Theorem \ref{thmsolveyamabe} is such that the integral of the scalar curvature of the metric $\hat{h}=\phi^2h$ is positive, since by the definition of the Yamabe constant, $\int_{\mathcal{N}} \hat{R} d\mu_{\hat{h}} \geq \mathcal{Y}(\mathcal{N},[h])\|\phi\|^2_{L^4(\mathcal{N})}>0$. This remark will be used in the next section (see Proposition \ref{propfund}).

\section{Proof of Theorem C}

\indent We start with two propositions that are the key results needed for the proof of Theorem C. The first one is a version of the Gauss-Bonnet-Chern formula.
\begin{lemm} \label{lemmgaussbonnetchern}
 Let $(\mathcal{N}^4,h)$ be an oriented associated singular Einstein manifold to a compact static triple $(M^3,g,V)$ with positive scalar curvature. \\
 \indent Given $u\in C^{4}(M)\cap C^{\infty}(int(M))$, consider the Riemannian metric $\tilde{h}=u^2h=u^2(V^2d\theta^2+g)$ on $\Omega$. Then the singular space $(\mathcal{N}^4,\tilde{h})$ is such that 
 \begin{equation} \label{eqgaussbonnetchern}
  \frac{1}{48}\int_{\mathcal{N}} \tilde{R}^2d\mu_{\tilde{h}} + \int_{\mathcal{N}}|\tilde{W}^{+}|_{\tilde{h}}^2d\mu_{\tilde{h}} = 8\pi^2\sum_{i=1}^{r}k_i + \frac{1}{4}\int_{\mathcal{N}}|\mathring{\widetilde{Ric}}|_{\tilde{h}}^2d\mu_{\tilde{h}},
 \end{equation}
 \noindent where $k_i$ denotes the constant value of $|\nabla V|$ on a connected component of $\partial M=\partial_{1} M\cup\ldots\cup\partial_{r} M$ for all $i=1,\ldots,r$.
\end{lemm}
\begin{proof}
\indent First, we study the behaviour of the metric $\tilde{h}=u^2h=u^2(V^2d\theta^2+g)$ near its singular set $\partial M$. In order to do this, we introduce Fermi coordinates on $M$ with respect to the metric $\tilde{g} = u^{2}g$. Since by hypothesis this is a $C^4$ metric on $M$, the normal exponential map of $\partial M$ is only of class $C^3$. \\
\indent Let $s$ and $\tilde{s}$ denote the distance function to the boundary $\partial M$ in $(M^3,g)$ and $(M^3,\tilde{g})$, respectively. Let $x=(x^2,x^3)$ be local coordinates in $\partial M$. We have $\tilde{s}(s,x)=u(0,x)s+O(s^2)$. Since $V(0,x)=0$ and $\partial_{s}V(0,x)=k_{i}$ at a point $(0,x)=(0,x^2,x^3)$ belonging to the connected component $\partial_i M$ (Proposition \ref{propexpansao}), we also have
\begin{equation*}
  (uV)(\tilde{s},x^2,x^3) = (u(0,x)+ O(\tilde{s}))(\partial_{\tilde{s}}V(0,x^2,x^3)\tilde{s}+O(\tilde{s}^2))= k_{i}\tilde{s}+O(\tilde{s}^2).
\end{equation*}
\indent Using the local coordinates $(x^{0}=\theta,x^1=\tilde{s},x^2,x^3)$ near $\partial_{i} M$, we conclude that $\tilde{h}=(uV)^2d\theta^2+\tilde{g}$ has the expansion
\begin{multline*}
 \tilde{h}_{ij}(\tilde{s},x^2,x^3)=d\tilde{s}^2 + k_{i}^{2} \tilde{s}^2 d\theta^2 + \sum_{i,j=2}^{3}\tilde{g}_{ij}(0,x^2,x^3)dx^i dx^j + \\ + 2\sum_{i,j=2}^{3}\tilde{A}_{ij}(0,x^2,x^3)\tilde{s}dx^i dx^j + \sum_{i,j=2}^{3} \tilde{G}_{ij}(0,x^2,x^3)\tilde{s}^2dx^idx^j + E,
\end{multline*}
\noindent where $\tilde{A}(0,x^2,x^3)$ is the second fundamental form of $\partial M$ in $(M^3,\tilde{g}=u^2g)$ at $x=(0,x^2,x^3)$. Moreover, the remainder $E$ satisfies $|E|\leq C\tilde{s}^3$ for some constant $C>0$ and $E(\partial_{\tilde{s}},\partial_{x^{A}})=0$ for all $A=0,1,2,3$. \\
\indent Thus, $\tilde{h}$ is singular in the sense of \cite{LiuShe}. Accordingly to \cite{LiuShe} (compare with \cite{AtiLeB}), a Gauss-Bonnet-Chern formula holds for the singular space $(\mathcal{N}^4,\tilde{h})$ and is given by
\begin{multline*}
 \frac{1}{48}\int_{\mathcal{N}} \tilde{R}^2d\mu_{\tilde{h}} + \frac{1}{8}\int_{\mathcal{N}}|\tilde{W}|_{\tilde{h}}^2d\mu_{\tilde{h}} - \frac{1}{4}\int_{\mathcal{N}}|\mathring{\widetilde{Ric}}|_{\tilde{h}}^2d\mu_{\tilde{h}}= \\ = 4\pi^2(\chi(\mathcal{N}) - \chi(\partial M) + \sum_{i=1}^{r} k_i\chi(\partial_{i} M)).
\end{multline*}
\indent In the above formula $|\tilde{W}|_{\tilde{h}}$, denotes the norm of the Weyl tensor $\tilde{W}$. Considering the self-dual/anti-self-dual decomposition $\tilde{W}^{+}+\tilde{W}^{-}$ of $\tilde{W}$ viewed as a symmetric endomorphism of the space of two-forms, we have $|\tilde{W}|^2_{\tilde{h}} = 4(|\tilde{W}^{+}|_{\tilde{h}}^2+|\tilde{W}^{-}|_{\tilde{h}}^2)$. \\
\indent As observed in Section $6$, $\mathcal{N}^4$ and the fixed point set of the $S^1$ action on $\mathcal{N}$, i.e, $\partial M$, have the same Euler characteristic. Moreover, since the map $\theta\mapsto -\theta$ is an orientation inverting isometry of $(\mathcal{N}^4,\tilde{h})$, $|\tilde{W}|_{\tilde{h}}^2=8|\tilde{W}^{+}|_{\tilde{h}}^2$. By Theorem B, $iii)$, each connected component of $\partial M$ is a two-sphere. Formula (\ref{eqgaussbonnetchern}) follows.
\end{proof}
\begin{rmk}
 The associated singular Einstein manifold $(\mathcal{N}^4,h=V^2d\theta^2+g)$ itself is such that
 \begin{equation*}
  \frac{1}{48}\int_{\mathcal{N}}R_{h}^2d\mu_{h}=2\pi\sum_{i=1}^{r}|\partial_i M|     \quad \text{and} \quad \int_{\mathcal{N}}|W^{+}|_{h}^2d\mu_{h}=2\pi\int_{M}|\mathring{Ric}_{g}|^2_{g}Vd\mu_{g}.
 \end{equation*}
\indent Thus, formula (\ref{eqgaussbonnetchern}) for $(\mathcal{N}^4,h)$ is precisely the fundamental formula (\ref{eqgaussbonnetchern1}) for $(M^3,g,V)$ (see Proposition \ref{propshenform}). This curious observation can be used to give a different proof of the above lemma using only the conformal invariance of the Weyl term in (\ref{eqgaussbonnetchern}) and the conformal transformations laws of the $Q$-curvature (for the definition and properties of the $Q$-curvature, see for example \cite{Gur2}). 
\end{rmk}
\indent The second key proposition is a consequence of the existence result Theorem \ref{thmsolveyamabe} (compare with \cite{Gur}, Proposition 3.5).  
\begin{prop} \label{propfund}
 Let $(\mathcal{N}^4,h)$ be the associated singular Einstein manifold to a compact static triple $(M^3,g,V)$ with positive scalar curvature. Assume that $(M^3,g,V)$ is not equivalent to the standard hemisphere. \\
 \indent Then
 \begin{equation*}
  \inf_{\substack{\tilde{h}=u^2h,\, u > 0 \\ u\in C^{\infty}(M)}} \int_{\mathcal{N}} \tilde{R}^2d\mu_{\tilde{h}} \leq 24\int_{\mathcal{N}}|W^{+}|^2_{h}d\mu_{h}.
 \end{equation*}
 \indent Moreover, assume the equality holds. Let $\phi$ be the solution to the modified Yamabe problem given in Theorem \ref{thmsolveyamabe}. Then $\phi\in C^{\infty}(M)$ and the metric $\hat{h}=\phi^{2}h$ on $\Omega$ is such that $\tilde{R} = 2\sqrt{6}|\tilde{W}^{+}|_{\tilde{h}}$ is a positive constant.
\end{prop}
\begin{proof} By Theorem \ref{thmsolveyamabe}, there exists a Riemannian metric $\hat{h}=\phi^2h$ on $\Omega=\mathcal{N}\setminus\partial M$, where $\phi\in C^{0}(M)\cap C^{2,\alpha}(int(M))$ is positive, such that
 \begin{equation*}
  \hat{R} - 2\sqrt{6}|\hat{W}^{+}|_{\hat{h}}
 \end{equation*}
 \noindent is a constant $\lambda\leq 0$. Multiplying by $\hat{R}$ and integrating on $\mathcal{N}$, we obtain
 \begin{align*}
  \int_{\mathcal{N}} \hat{R}^2d\mu_{\hat{h}} & = 2\sqrt{6}\int_{\mathcal{N}}\hat{R}|\hat{W}^{+}|_{\hat{h}}d\mu_{\hat{h}} + \lambda \int_{\mathcal{N}} \hat{R}d\mu_{\tilde{h}} \\
              & \leq 2\sqrt{6}(\int_{\mathcal{N}}\hat{R}^2d\mu_{\hat{h}})^{1/2}(\int_{\mathcal{N}}|\hat{W}^{+}|_{\hat{h}}^2d\mu_{\hat{h}})^{1/2} + \lambda \int_{\mathcal{N}} \hat{R}d\mu_{\hat{h}}.
 \end{align*}
 \indent As remarked in the end of the previous section, $\int_{\mathcal{N}} \hat{R}d\mu_{\hat{h}} > 0$. Therefore
 \begin{equation*}
  \int_{\mathcal{N}} \hat{R}^2d\mu_{\hat{h}} \leq 24 \int_{\mathcal{N}}|\hat{W}^{+}|^2_{\hat{h}}d\mu_{\hat{h}} = \int_{\mathcal{N}} |W^{+}|_{h}^2d\mu_{h},
 \end{equation*}
 \noindent and equality holds if and only if $\lambda=0$. The last equality follows from the conformal invariance of $|W^{+}|_{h}^2d\mu_{h}$ in dimension four. \\
 \indent Since an approximation argument implies that
 \begin{equation*} \label{inequality}
  \inf_{\substack{\tilde{h}=u^2h,\, u > 0 \\ u\in C^{\infty}(M)}} \int_{\mathcal{N}} \tilde{R}^2d\mu_{\tilde{h}} = \nonumber \inf_{\substack{\tilde{h}=u^2h,\, u > 0 \\ u\in C^{0}(M)\cap C^{2,\alpha}(int(M))}} \int_{\mathcal{N}} \tilde{R}^2d\mu_{\tilde{h}},
 \end{equation*}
 \noindent we conclude that 
 \begin{equation*}
  \inf_{\substack{\tilde{h}=u^2h, u > 0 \\ u\in C^{\infty}(M)}} \int_{\mathcal{N}} \tilde{R}^2d\mu_{\tilde{h}} \leq \int_{\mathcal{N}} \hat{R}^2d\mu_{\hat{h}}\leq \int_{\mathcal{N}} |W^{+}|_{h}^2d\mu_{h}. 
 \end{equation*}
 \indent If equality holds, then $\hat{R}=2\sqrt{6}|\hat{W}^{+}|_{\hat{h}}$ and the infimum of the functional $\hat{h} \mapsto \int_{\mathcal{N}} \hat{R}^2d\mu_{\hat{h}}$ on the set $\{\hat{h}=\psi^2h; \,\psi\in C^{0}(M)\cap C^{2,\alpha}(int(M)),  \psi >0\}$ is attained by $\phi$. Considering variations $\hat{h}_{t}=(\phi+t\phi^{-1}v)^2h$, $v\in C_{0}^{\infty}(int(M))$, one can show that $\hat{R}$ must be constant (see \cite{Bes}, 4.67). Therefore $|\hat{W}^{+}|_{\hat{h}}$ is constant and $\phi$ satisfies the equation 
 \begin{equation*}
  -6div_{g}(V\phi) + 12V\phi = \hat{R}V\phi^3.
 \end{equation*}
 \indent In Appendix C, we argue that $\phi$ is smooth up to the boundary. The proposition follows. 
\end{proof}

\indent We have now all ingredients we need to prove Theorem C.

\begin{thm}
 Let $(M^3,g,V)$ be a compact simply connected static triple with scalar curvature $6$. One of the following alternatives holds:
 \begin{itemize}
  \item[$i)$] $(M^3,g,V)$ is equivalent to the standard hemisphere; or
  \item[$ii)$] $(M^3,g,V)$ is equivalent to the standard cylinder; or
  \item[$iii)$] Denoting by $\partial M=\partial_1 M \cup \ldots \cup \partial_r M$  the connected components of the boundary and by $k_{i}$ the value of $|\nabla V|$ on $\partial_{i} M$, the following inequality holds
  \begin{equation*}
   \sum_{i=1}^{r} k_{i} |\partial_i M| < \frac{4\pi}{3}\sum_{i=1}^{r} k_{i}.
  \end{equation*}
 \end{itemize}  
\end{thm}

\begin{proof}
 Let $(\mathcal{N}^{4},h)$ be the associated singular Einstein manifold to $(M^3,g,V)$. Assume $(M^3,g,V)$ is not equivalent to the standard hemisphere. As a consequence of Proposition \ref{propfund}, there exists a positive function $u\in C^{\infty}(M)$ such that $(\mathcal{N}^4,\tilde{h}=u^2h)$ satisfies
 \begin{equation*}
    \frac{1}{24}\int_{\mathcal{N}} \tilde{R}^2d\mu_{\tilde{h}} \leq \int_{\mathcal{N}}|\tilde{W}^{+}|^2_{\tilde{h}}d\mu_{\tilde{h}} = \int_{\mathcal{N}} |W^{+}|_{h}^2d\mu_{h}.
 \end{equation*}
 \indent Moreover, we can assume equality holds if and only if $\tilde{R}=2\sqrt{6}|\tilde{W}^{+}|_{\tilde{h}}$ is a positive constant. \\ 
 \indent Combining the above inequality with the Gauss-Bonnet-Chern formula (\ref{eqgaussbonnetchern}) for $(\mathcal{N}^4,\tilde{h})$, we conclude that
 \begin{equation} \label{desimplicaeinstein}
  \frac{3}{2}\int_{\mathcal{N}}|W^{+}|_{h}^2d\mu_h  \geq  8\pi^2\sum_{i=1}^{r}k_i + \frac{1}{4}\int_{M}|\mathring{\widetilde{Ric}}|_{\tilde{h}}^2 d\mu_{\tilde{h}} \geq 8\pi^2\sum_{i=1}^{r}k_i.
 \end{equation}
 \indent Now, $(\mathcal{N}^4,h=V^2d\theta^2+g)$ satisfies
 \begin{equation*}
   \int_{\mathcal{N}}|W^{+}|_{h}^2d\mu_{h}=2\pi\int_{M}|\mathring{Ric}_{g}|^2_{g}Vd\mu_{g}.
 \end{equation*}
 \indent By the fundamental formula (\ref{eqgaussbonnetchern1}) for $(M^3,g,V)$ (see Proposition \ref{propshenform} and Theorem B, $iii)$), we also have
 \begin{equation*}
 \sum_{i=1}^{r} k_i |\partial_{i} M| + \int_{M} |\mathring{Ric}_g|_{g}^2 Vd\mu_{g} = 4\pi\sum_{r=1}^{r}k_i.
 \end{equation*}
 \indent Therefore
 \begin{equation*}
  \sum_{i=1}^{r} k_{i} |\partial_i M| \leq \frac{4\pi}{3}\sum_{i=1}^{r} k_{i}.
 \end{equation*}
 \indent If equality holds above, the conformally related metric $\tilde{h}=u^2h$ must be Einstein (by (\ref{desimplicaeinstein})) and $\tilde{R}=2\sqrt{6}|\tilde{W}^{+}|_{\tilde{h}}$ must be a positive constant. As shown in \cite{Gur}, using the formulas relating the divergence of the Weyl tensors of conformally related metrics one can then conclude that $u$ must be constant. Thus, the metric $h=V^2d\theta^2+g$ also satisfies $R_{h}=2\sqrt{6}|W^{+}|_{h}$, that is, $|\mathring{Ric}_g|_{g}=\sqrt{6}$. And then Theorem A implies that $(M^3,g,V)$ is the standard cylinder.
\end{proof}

\appendix

\section{An inequality for transverse trace-free tensors}

\indent We prove a generalization of the refined Kato inequality for Codazzi tensors in three-manifolds (see \cite{Bour2}). The proof is similar to \cite{Yan}, Lemma 5.1. We hope this inequality will be useful in other contexts.

\begin{prop}
 Let $(M^3,g)$ be a Riemannian three-manifold and $T$ a symmetric $2$-tensor which has constant trace and zero divergence. Denote by $C$ the $3$-tensor defined by 
 \begin{equation*}
  C(X,Y,Z) = \nabla T(X,Y,Z) - \nabla T(X,Z,Y) \quad \text{for all} \quad X,Y,Z\in \mathcal{X}(M).
 \end{equation*}
  Then, at every $p\in M$ such that $|T|(p)\neq 0$,
 \begin{equation*}
  |\nabla |T||^2 \leq  \frac{3}{5}\left( |\nabla T|^2 + \frac{|C|^2}{2}\right).
 \end{equation*} 
\end{prop}

\begin{proof}
 Let $e_1,e_2,e_3$ be a local orthonormal referential that diagonalizes $T$ at a point $p\in M$ where $|T|\neq 0$. Then $|\nabla |T|^2| = 4|T|^2|\nabla|T||^2$ implies
 \begin{equation*}
  |\nabla|T||^2 = \frac{1}{|T|^2}\sum_{k=1}^{3}(\sum_{i,j=1}^{3}T_{ij}T_{ij;k} )^2 = \frac{1}{|T|^2}\sum_{k=1}^{3}(\sum_{i=1}^{3}T_{ii}T_{ii;k} )^2.
 \end{equation*}
 \indent By Cauchy-Schwartz inequality,
 \begin{equation} \label{apdescs}
  |\nabla|T||^2 \leq \frac{1}{|T|^2}\sum_{k=1}^{3}(\sum_{i=1}^{3}T_{ii}^2)(\sum_{i=1}^{3}T_{ii;k}^2)=\sum_{i,k=1}^{3}T_{ii;k}^2.
 \end{equation}
 \indent Let $A,B,C\in\{1,2,3\}$ denote different indexes. Then
 \begin{equation} \label{apidentcott}
  C_{BBA}+C_{CCA} = \sum_{i=1}^3C_{ii;A} = \sum_{i=1}^3T_{ii;A} - \sum_{i=1}^3T_{iA;i} = d(tr(T))_A - div(T)_{A} =0,
 \end{equation}
\noindent by hypothesis. This implies 
 \begin{eqnarray*}
  T_{AA;A} & = & - T_{BB;A} - T_{CC:A} = - T_{BA;B} - T_{CA;C},  \\
  T_{BB;A} & = & T_{BA;B} + C_{BBA}, \\
  T_{CC;A} & = & T_{CA;C} + C_{CCA} = T_{CA;C} - C_{BBA}.
 \end{eqnarray*} 
 \indent Thus, rewriting inequality (\ref{apdescs}) as
 \begin{equation*}
  |\nabla |T||^2 \leq \sum_{i=1}^3(T_{11;i}^2+T_{22;i}^2+T_{33;i}^2),
 \end{equation*}
 \noindent each of the above summand is of the form $(a+b)^2+(a-c)^2+(b+c)^2$ for some $a,b,c\in \mathbb{R}$. Because of the algebraic identity 
 \begin{equation*}
  (a+b)^2+(a-c)^2+(b+c)^2 + (a-b+c)^2 = 3(a^2+b^2+c^2),
 \end{equation*}
 \noindent we then conclude that 
 \begin{align}
 \nonumber |\nabla|T||^2 & \leq  3(T_{21;2}^2+T_{31;3}^2 + C_{221}^2) + 3(T_{12;1}^2 + C_{112}^2+T_{32;3}^2) \\
 \nonumber               & \quad +3(C_{223}^2+T_{13;1}^2+T_{23;2}^2) \\
 \nonumber               & =   3\sum_{i\neq k}T_{ik;i}^2 + \frac{3}{2}((C_{221}^2+C_{331}^2)+ (C_{112}^2+C^2_{332}) + (C_{223}^2 + C_{113}^2) ) \\
 \nonumber               & =  \frac{3}{2}( \sum_{i\neq k}T_{ik;i}^2 + \sum_{i\neq k}T_{ki;i}^2 + \sum_{i\neq k}C_{iik}^2 ), \\
                         & = \frac{3}{2}( \sum_{i\neq k}T_{ik;i}^2 + \sum_{i\neq k}T_{ki;i}^2 + \frac{1}{2}(\sum_{i\neq k}C_{iik}^2+\sum_{i\neq k}C_{iki}^2  )), \label{apdesfund}
 \end{align}
\noindent where in the last three lines we used (\ref{apidentcott}), the symmetry of $T$ and the antisymmetry of $C$ in the last two variables. The proposition follows by combining inequalities (\ref{apdescs}) and (\ref{apdesfund}).
\end{proof}

\section{Solutions to an equation of type $\mathring{Hess V}=0$}

\indent Let $(M^n,g)$ be a Riemannian manifold admitting a non-constant solution to the equation $Hess V = (\Delta V/n)g$. These manifolds have been classified by K\"uhnel \cite{Kuhn}. In this section, we consider a particular case of this equation on compact surfaces $(\Sigma^2,g)$,
\begin{equation*}
 Hess V = \frac{1}{2}(K-3\epsilon) Vg,
\end{equation*}
\noindent where $K$ is the Gaussian curvature of $(\Sigma,g)$ and $\epsilon\in\{-1,0,1\}$. It has appeared in the proof of Proposition \ref{flow} in Section 8. \\
\indent For completeness, we prove here the classification of positive solutions of this equations. The argument follows \cite{Gal}, Lemma 3.

\begin{prop}
 Let $(\Sigma^2,g)$ be a closed surface with Gaussian curvature $K$ and $\epsilon\in\{-1,0,1\}$. If there exists a positive solution $V\in C^{\infty}(M)$ of the equation
\begin{equation} \label{apeqaux}
  Hess_{\Sigma} V = \frac{1}{2}(K - 3\epsilon)Vg,
 \end{equation}
 \noindent then $V$ is constant and $\Sigma$ has constant Gaussian curvature $K=3\epsilon$. 
\end{prop}

\begin{proof}
 \indent Taking the divergence of (\ref{apeqaux}), we obtain the following integrability condition,
 \begin{equation} \label{apintcont}
  d((K-\epsilon)V^3)=0.
 \end{equation}
 \indent Therefore $K=cV^{-3}+\epsilon$ for some constant $c$. Integrating the trace of (\ref{apeqaux}), we have that $c$ is positive, zero or negative if $\epsilon$ equals $1, 0$ or $-1$, respectively. We analyse each case separately. \\
 \indent If $\epsilon=0$, then $c=0$ and therefore $K\equiv 0$ and $Hess V=0$. The theorem follows.\\
 \indent If $\epsilon= -1$, then $c<0$ and therefore $K<-1$. Using this and the integrability condition, we conclude that a point $p\in\Sigma$ is a critical point of $K$ if and only if it is a critical point of $V$. Moreover, 
 \begin{equation*}
  \Delta K = c\left(\frac{12}{V^5}|\nabla V|^2 - \frac{3}{V^4}\Delta V\right) =  c\left(\frac{12}{V^5}|\nabla V|^2 - \frac{3}{V^3}(K+3)\right).
 \end{equation*}
 \indent Hence, at a point $p_0\in \Sigma$ of absolute minimum of $K$,
 \begin{equation*}
  0\leq \Delta K(p_0) = -\frac{3c}{V^3(p_0)}(K(p_0)+3) \Rightarrow K(p_0)\geq -3,
 \end{equation*}
 \noindent since $c<0$ and $V$ is a positive function. Similarly, at a point $p_1\in\Sigma$ of absolute maximum of $K$, $K(p_1)\leq -3$. Therefore $K \equiv -3$ and $Hess V =0$. The result follows. \\
 \indent  If $\epsilon = 1$, then $c > 0$ and therefore $K > 1$. We can assume $\Sigma$ is homeomorphic to the sphere $S^2$. \\
 \indent Suppose, by contradiction, that $V$ is not constant. By \cite{Kuhn}, the existence of a non-constant solution $V$ to the equation 
 \begin{equation*}
  \mathring{Hess} V = 0
 \end{equation*}
 \noindent in a closed surface implies that $V$ has precisely two non-degenerated critical points (the point of minimum $p_0$ and the point of maximum $p_1$), that $V$ depends only on the distance to a critical point, and that the metric on $\Sigma\setminus\{p_0,p_1\}$ can be written in the warped-product form
 \begin{equation*}
  g = du^2 + \left(\frac{V'(u)}{V''(0)}\right)^{2} d\theta^2,
 \end{equation*}
 where $u\in(0,u_{*})$ is the distance to $p_0$ and $\theta$ is a $2\pi$-periodic variable (see \cite{Kuhn}). $V$ is an increasing function of $u$ defined on the interval $[0,u_{*}]$.\\
 \indent For the above metric, $Hess_g V = V'' g$, so equation (\ref{apeqaux}) is equivalent to 
 \begin{equation} \label{apeqaux2}
  2V'' = (K-3)V = \frac{c}{V^{2}} - 2V.
 \end{equation} 
 \indent This second order differential equation can be expressed in the Hamiltonian form using the energy
 \begin{equation*}
  H(x,y)=y^2 + \frac{c}{x} + x^2
 \end{equation*}
 \noindent for $x=V > 0$, $y=V'$. Sketching the level sets, one can verify that the constant solution of (\ref{apeqaux2}) corresponds to the point $(x,y)=((c/2)^{1/3},0)$ and that the solution $V$ of (\ref{apeqaux2}) originated in our problem is such that $V(0)<(c/2)^{1/3} < V(u_{*}) $. \\
 \indent Now we use the Gauss-Bonnet theorem and the explicit expression of the metric to compute:
 \begin{equation*}
  4\pi = \int_{\Sigma}Kd\Sigma = \int_{0}^{2\pi}\int_{0}^{u_{*}}\left[1+\frac{c}{V^3} \right]\frac{V'}{V''(0)}dud\theta = \frac{2\pi}{V''(0)}\left[V-\frac{c}{2V^2}\right]^{u_{*}}_{0}.
 \end{equation*}
 \indent By (\ref{apeqaux2}), we conclude that $c = 2V^3(u_{*})$. And this is a contradiction. 
\end{proof}

\begin{prop}
  Let $(\Sigma^2,g)$ be a compact surface with non-empty boundary. Let $K$ denote its Gaussian curvature and let $\epsilon\in\{-1,0,1\}$. Assume that there exists a non-trivial non-negative solution $V\in C^{\infty}(M)$ to the problem
\begin{equation*}
  Hess_{\Sigma} V = \frac{1}{2}(K - 3\epsilon)Vg  \quad \text{in} \quad \Sigma, \quad V=0 \quad \text{in} \quad \partial \Sigma.
 \end{equation*}
 \indent Then $\epsilon=1$, $\Sigma$ is isometric to the round hemisphere $(S^{2}_{+},g_{can})$ with constant Gaussian curvature $K=1$ and $V$ is a static potential on $(\Sigma^2,g)$.
\end{prop}

\begin{proof}
 The integrability condition (\ref{apintcont}) and the hypothesis that $V$ vanishes on $\partial \Sigma$ implies that $K=\epsilon$ on $\Sigma$. Taking the trace of equation (\ref{apeqaux}), we obtain
 \begin{equation} \label{apeqtraco}
  \Delta V + 2\epsilon V = 0.
 \end{equation}
  \indent Hence, multiplying by $V$ and integrating by parts on $\Sigma$ using that $V$ vanishes on $\partial \Sigma$ one concludes that $\epsilon$ must be positive. Therefore $V\neq 0$ satisfies the static equations (\ref{apeqaux}) and (\ref{apeqtraco}), since $(\Sigma^2,g)$ has constant Gaussian curvature $K=\epsilon=1$.
\end{proof}

\section{Regularity of solutions to a degenerated elliptic problem}

\indent Given $(M^3,g,V)$ a compact orientable static triple, let $(\mathcal{N}^4,h)$ be its associated singular Einstein manifold. We want to study the regularity properties of a positive solution $\phi\in C^{0,\mu}(\mathcal{N})\cap W^{1,2}(\mathcal{N})\cap C^{2,\alpha}(int(M))$ to the Yamabe equation 
\begin{equation} \label{apeqdiv}
 -\Delta_{h} \phi + 2\phi = 2\phi^3 
\end{equation}
\noindent in the singular space $(\mathcal{N}^4,h)$. Equation (\ref{apeqdiv}) means that the conformally related metric $\hat{h}=\phi^2 h$ on $\mathcal{N}\setminus \partial M$ has constant scalar curvature $12$. \\
\indent The first observation is that $\phi\in C^{\infty}(int(M))$ by standard elliptic regularity \cite{GilTru}. The problem is to understand the behaviour of $\phi$ near $\partial M$. In order to analyse this point, we first deduce some formulas in an appropriated coordinate system near $\partial M$. \\
\indent Let $(x^{0}=s,x^1,x^2)$ be Fermi coordinates on a small tubular neighborhood $\mathcal{U}=\{p\in M;\, s(p)=d_{\partial M}(p)< \epsilon\}$ of $\partial M$, where $(x^{1},x^2)$ are local coordinates on $\partial M$. We write $M_{s}=\{p\in M;\, d_{\partial M}(p)=s\}$ and use the subscript $s$ to denote the geometric quantities related to $M_s$ in $(M,g)$. Using the fact that $\partial M$ is totally geodesic in $(M,g)$ and the expansion of $V$ near $\partial M$ (Proposition \ref{propexpansao}), we conclude that the Laplacian of the metric $h=V^2d\theta^2+g$ takes the form:
\begin{align*}
 \Delta_{h}\phi  & =  \Delta_g\phi + \frac{1}{V}g(\nabla^{g} V,\nabla^{g} \phi)  \\
     & = (\partial_{s}^2\phi + \Delta_{s} \phi + H_{s}\partial_s \phi) + \frac{\partial_{s}V}{V}\partial_s\phi + \frac{1}{V}g^{ij}\partial_{i}V\partial_{j}\phi \\
        &  = \partial_{s}^2\phi + \frac{1}{s}\partial_s \phi + \Delta_{0}\phi + Q(\phi),
\end{align*}
\noindent where, for indexes $i,j$ running through $\{1,2\}$ and indexes $A$ through $\{0,1,2\}$,
\begin{align*} 
 \Delta_{0}\phi & = g^{ij}(0,x)(\partial_{i}\partial_{j}\phi - \Gamma_{ij}^{k}(0,x)\partial_{k}\phi) \quad \text{and} \\
 Q(\phi) & = Q^{ij}s^2\partial_{i}\partial_{j}\phi + Q^{A}s\partial_A \phi \quad \text{for} \quad Q^{i,j}, Q^{A} \in C^{\infty}([0,\epsilon)\times \partial M).
\end{align*}
\indent Hence, 
\begin{equation} \label{apeeqdiv}
 \partial_s\phi^2 + \frac{1}{s}\partial_s\phi + \Delta_{0}\phi + Q(\phi) = 2\phi - 2\phi^{3}.
\end{equation}
\indent Now, considering the metric $\tilde{h}=s^{-2}h$ for some smooth positive function $s$ on $M\setminus\partial M$ that coincides with the distance function to $\partial M$ in $\mathcal{U}$, the function $\eta = s\phi$ satisfies
\begin{equation} \label{apeqconf}
 -\Delta_{\tilde{h}} \eta + \tilde{R}\eta= 2\eta^3,
\end{equation}
\indent In the above equation, $\tilde{R}$ denotes the scalar curvature of $(\mathcal{N},\tilde{h})$. Using the formulas relating the geometric quantities associated to conformal metrics (see \cite{Bes}), one observes that $\Delta_{\tilde{h}}\eta = s^2\Delta_{h}\eta - 2ds(\nabla^{h}\eta)$ and that $\tilde{R}\in C^{\infty}(M)$ satisfies $\tilde{R}=-12+O(s^2)$ as $s$ goes to zero. \\
\indent Since $\eta$ is invariant under the $S^1$ action on $(\mathcal{N}^4,h)$, this equation reduces to a uniformly degenerated elliptic equation (in the sense of \cite{GraLee}) on the conformally compact manifold $(\overline{M}=M\setminus\partial M,\overline{g}=s^{-2}g)$. In fact, by the above formulas, we have
\begin{equation} \label{apeqyamabeconfcoord}
 s^2\partial_{s}\eta - s\partial_s \eta + s^2\Delta_s \eta + \tilde{Q}^{A}s\partial_{A}\eta = \frac{1}{6}\hat{R}\eta - 2\eta^3
\end{equation}
\noindent where $\Delta_{s}\eta = g^{ij}(\partial_{i}\partial_{j}\phi - \Gamma_{ij}^{k}\partial_{k}\phi)  $ for indexes $i,j$ running through the set $\{1,2\}$ and $\tilde{Q}^{A} \in C^{\infty}([0,\epsilon)\times\partial M)$ satisfies $\tilde{Q}^{A}=O(s^2)$ as $s$ goes to zero for all $A\in\{0,1,2\}$. Observing that $\eta=s\phi$, $\phi\in C^{0,\mu}(M)\cap C^{\infty}(int(M))$, Schauder type estimates for $\eta$ implies that $\phi=s^{-1}\eta$ is such that $s^{|\alpha|}\partial_{x}^{\alpha}\phi$ is uniformly bounded for every multi-index $\alpha$ (see \cite{GraLee}, Proposition 3.4). \\
\indent Accordingly to \cite{AkuCarMaz1}, one can go further and show that $i)$ $\phi$ is in fact conormal (i.e., the derivatives $s^{i}\partial_{s}^{i}\partial_{x^{1}}^{j}\partial_{x^{2}}^{k}\phi$ are bounded in $M$ for all non-negative integers $i,j,k$); and $ii)$ has an expansion in $s$ as a sum involving terms like $a_{pq}s^p \log^q(s)$ where $a_{pq}\in C^{\infty}(\partial M)$. The reader will find general theories about this kind of degenerated elliptic equation developed in \cite{Maz} (the ``edge-calculus") and \cite{AndChr} (we also refer the reader to \cite{AndChrFri}, where some theory is developed in a very concrete situation). \\
\indent Using the information that $\phi$ has an expansion in $s$ near $\partial M$, we compute it considering that $\phi$ satisfies
\begin{equation} \label{apeqyamabeconfcoord} 
 \partial_{s}^2\phi + \frac{1}{s}\partial_s\phi = F
\end{equation}
\noindent for $F = F(\phi) = 2\phi-2\phi^3-\Delta_{0}\phi - Q(\phi)$ smooth on $(0,\epsilon)\times \partial M$ and such that $s^{k}\partial_s^{(k)}F$ is bounded in $M$ for all positive integer $k$. \\
\indent For every $x\in \partial M$ and every $s\in(0,\epsilon)$, there exists $\zeta_s\in(0,\epsilon)$ such that $\phi(s,x)=\phi(0,x)+\partial_s\phi(\zeta_s,x)s$, by Taylor's Theorem. Since $\phi$ is smooth on the interior of $M$, from this formula for $x \mapsto \phi(0,x)$ it follows that $\phi$ restricted to $\partial M$ belongs to $C^{\infty}(\partial M)$. \\
\indent We will compute the expansion of $\phi$ using the classical method of separation of variables. We are grateful to Andr\'e Neves for suggesting this approach. \\
\indent Let $\lambda_0=0<\lambda_1\leq\lambda_2\leq \ldots$ denote the eigenvalues of the Laplacian $\Delta_0$ on $(\partial M,g_{\partial M})$ and let $\{\xi_{i}\}$ be a $L^2$-orthonormal basis of eigenfunctions,
\begin{equation*}
 \Delta_{0}\xi_i + \lambda_i \xi_i = 0.
\end{equation*}
\indent For each $s\in(0,\epsilon)$, we can write
\begin{equation*}
 \phi(s,x) = \sum_{i=0}^{+\infty} \alpha_{i}(s)\xi_{i}(x) \quad \text{where} \quad \alpha_{i}(s)=\int_{\partial M} \phi(s,x)\xi_{i}(x)d\mu_{0}(x).
\end{equation*}
\indent The convergence of the above series is to be understood in the space $L^2(\partial M,d\mu_b)$. Since $\phi$ is smooth in the interior of $M$ and continuous up to $\partial M$, we have that $\alpha_i\in C^{0}([0,\epsilon))\cap C^{\infty}((0,\epsilon))$ and, in fact, for all $s\in(0,\epsilon)$,
\begin{equation*}
 \partial_{s}^{(k)}\phi(s,x) = \sum_{i=0}^{+\infty} \alpha^{(k)}_{i}(s)\xi_{i}(x) \quad \text{where} \quad \alpha^{(k)}(s) = \frac{d^k\alpha}{ds^k}(s).
\end{equation*}
\indent Integrating (\ref{apeqyamabeconfcoord}) against each $\xi_i$ for each fixed $s\in(0,\epsilon)$ we conclude that each $\alpha_i\in C^{0}([0,\epsilon))\cap C^{\infty}((0,\epsilon))$ satisfies an equation of the form
\begin{equation} \label{apeqmodelo}
 \alpha'' + \frac{1}{s}\alpha' - \lambda\alpha = F
\end{equation}
\noindent for some constant $\lambda$ and a function $F\in C^{0}([0,s))\cap C^{\infty}((0,s))$ depending on $\phi$. In fact, $\alpha_i$ satisfies (\ref{apeqmodelo}) with $\lambda=\lambda_{i}+2$ and
\begin{equation*}
 F(s)=F_i(\phi)(s) = \int_{\partial M} (-2\phi^3(s,x)-Q(\phi)(s,x))\xi_i(x)d\mu_{0}(x).
\end{equation*}
\begin{lemm}
 Let $k$ be a non-negative integer. If $\alpha\in C^{0}([0,\epsilon))\cap C^{\infty}((0,\epsilon))$ is a solution to (\ref{apeqmodelo}) for some  $F\in C^{k}([0,\epsilon))\cap C^{\infty}((0,\epsilon))$, then $\alpha\in C^{k+2}([0,\epsilon))$.
\end{lemm}
\begin{proof}
 We first explain the argument for $k=0$. We have
\begin{equation}
 (s\alpha')' = \lambda s\alpha + sF.
\end{equation}
 \indent Integrating from $\delta>0$ to $s<\epsilon$ we obtain
\begin{equation*}
 s\alpha'(s) - \delta\alpha'(\delta) = \lambda\int_{\delta}^{s}t\alpha(t)dt + \int_{\delta}^{s}tF(t)dt.
\end{equation*} 
\indent Since $\alpha$ is bounded, $\lim_{\delta\rightarrow 0} \delta\alpha'(\delta)=0$ (otherwise, $\alpha(s)$ would grow like $|\log(s)|$ as $s$ goes to zero). Since $\alpha$ and $F$ are bounded, we can take the limit as $\delta$ goes to zero and conclude that, for all $s\in(0,\delta)$,
\begin{equation}\label{apprimeirader}
 \alpha'(s)=\frac{\lambda}{s}\int_{0}^{s}t\alpha(t)dt + \frac{1}{s}\int_{0}^{s}tF(t)dt.
\end{equation}
\indent In particular, $\lim_{s\rightarrow 0}\alpha'(s)= 0$ because $\alpha$ and $tF$ are continuous. Moreover for every $0<\delta<s<\epsilon$, we have
\begin{equation*}
 \frac{\alpha(s)-\alpha(\delta)}{s-\delta} = \frac{\lambda}{s-\delta}\int_{\delta}^{s}\frac{1}{u}\int_{0}^{u} t\alpha(t)dt du + \frac{1}{s-\delta}\int_{\delta}^{s}\frac{1}{u}\int_{0}^{u} tF(t)dt du 
\end{equation*}
\noindent so that taking the limit as $\delta$ goes to zero we obtain
\begin{equation*}
 \frac{\alpha(s)-\alpha(0)}{s} =\frac{\lambda}{s}\int_{0}^{s}\frac{1}{u}\int_{0}^{u} t\alpha(t)dt du + \frac{1}{s}\int_{0}^{s}\frac{1}{u}\int_{0}^{u} tF(t)dt du = O(s).
\end{equation*}
\indent It follows that $\alpha$ is differentiable at the origin and satisfies $\alpha'(0)=0$. Thus, $\alpha\in C^{1}([0,\epsilon))$.\\
\indent For every $s\in(0,s)$, integration by parts now gives
\begin{equation*}
 \int_{0}^{s} t \alpha(t) dt = \frac{s^{2}}{2}\alpha(s) - \frac{1}{2}\int_{0}^{s}t^{2}\alpha'(t)dt.
\end{equation*}
\indent We claim that a similar formula also holds for the integral of $tF(t)$. In fact, observe that  $tF'(t)$ is bounded (as before, we have $\lim_{s\rightarrow 0}sF'(s)=0$). For every $0<\delta<s<\epsilon$, integration by parts gives
\begin{equation*}
 \int_{\delta}^{s}tF(t)dt = \left(\frac{s^{2}}{2}F(s)- \frac{\delta^2}{2}F(\delta)\right) - \frac{1}{2}\int_{\delta}^{s}t^{2}F'(t)dt.
\end{equation*} 
\indent Using that $F$ is continuous and $tF'(t)$ is bounded, we can make $\delta$ go to zero to obtain
\begin{equation*}
 \int_{0}^{s}tF(t)dt = \frac{s^{2}}{2}F(s) - \frac{1}{2}\int_{0}^{s}t^{2}F'(t)dt,
\end{equation*}
\noindent as claimed. \\
\indent Differentiating (\ref{apprimeirader}) at $s\in(0,\epsilon)$, we obtain
\begin{align*}
 \alpha''(s) & = -\frac{\lambda}{s^2}\int_{0}^{s}t\alpha(t)dt + \lambda\alpha(s) - \frac{1}{s^2}\int_{0}^{s}tF(t)dt + F(s) \\
          & =  \frac{\lambda}{2}\alpha(s) + \frac{1}{2}F(s) + \frac{1}{2s^2}\int_0^{s}t^2(\lambda\alpha'(t)+ F'(t))dt.
\end{align*}
\indent Observe that, as $s$ goes to zero,
\begin{equation*}
 |\frac{1}{2s^2}\int_0^{s}(t^2(\lambda\alpha'(t)+ F'(t))dt | \leq C(\sup_{t\in(0,s)}|t\alpha'(t)|+ \sup_{t\in(0,s)}|tF'(t)|) \rightarrow 0
\end{equation*}
\indent Therefore, by continuity of $\alpha$ and $F$, $\lim_{s\rightarrow 0} \alpha''(s) = (\lambda\alpha(0) + F(0))/2$ and, similarly, we obtain
\begin{equation*}
 \frac{\alpha'(s)-\alpha'(0)}{s} =  \frac{\alpha'(s)}{s} = \frac{1}{2s}\int_{0}^{s}( \lambda\alpha(t) + F(t)) dt + o(1),
\end{equation*}
\noindent so that $\alpha$ is differentiable at $s=0$ with $\alpha''(0)=(\lambda\alpha(0) + F(0))/2$. Thus, $\alpha\in C^{2}([0,\epsilon))$ as we wanted to prove. \\
\indent For the general case, one can combine successive derivatives of (\ref{apprimeirader}) and integration by parts to prove by induction that for every $i=0,1,\ldots,k+2$ and every $s\in(0,\epsilon)$, 
\begin{equation*}
 \alpha^{(i)}(s) = \lambda \sum_{j=0}^{i-2}a_j\alpha^{(j)}(s) + \sum_{j=0}^{i-2}b_jF^{(j)}(s) + \frac{c_i}{t^{i}}\int_{0}^{s} t^{i}(\alpha^{(i-1)}(t)+F^{(i-1)}(t))dt
\end{equation*} 
\noindent for some constants $a_{j},b_{j}$ and $c_{i}$. Moreover, the last term behaves as $o(1)$ as $s$ goes to zero, $\alpha^{(i)}$ is differentiable at the origin and 
\begin{equation*}
 \alpha^{(i)}(0)=\lim_{s\rightarrow 0}\alpha^{(i)}(s) = \lambda a_{i}\sum_{j=0}^{i-2}\alpha^{(j)}(0)++ b_{i} \sum_{j=0}^{i-2}F^{(j)}(0)
\end{equation*}
\noindent for every $i=0,1,\ldots,k+1$. The result follows.
\end{proof}

\indent To prove the claimed regularity for the solution $\phi\in C^{0}(M)\cap C^{\infty}(int(M))$ (see the end of Proposition \ref{propfund}), we have to observe that for every non-negative integer $k$, $\alpha_i\in C^{k}([0,\epsilon))$ for all non-negative $i$ implies that $F=F(\phi) \in C^{k}(M)\cap C^{\infty}(int M)$ and apply induction using the above lemma.

\bibliographystyle{amsbook}

\end{document}